\documentclass[12pt,reqno]{amsart}

\usepackage{amssymb,graphicx, amsmath, amsthm, MnSymbol, array}
\usepackage{xcolor}
\usepackage{tikz-cd}

\newtheorem{theorem}{Theorem}[section]
\newtheorem{lemma}[theorem]{Lemma}
\newtheorem{proposition}[theorem]{Proposition}
\newtheorem{corollary}[theorem]{Corollary}

\newtheorem{claim}[theorem]{Claim}
\theoremstyle{definition}
\newtheorem{definition}[theorem]{Definition}
\newtheorem{example}[theorem]{Example}

\newtheorem{remark}[theorem]{Remark}

\newenvironment{proofclaim}{\paragraph{\emph{Proof of the Claim}.}}{\hfill$\qed$\\}

\def\bal{\boldsymbol{\mathit{ba}\ell}}

\def\ubal{\boldsymbol{\mathit{uba}\ell}}
\def\dbal{\boldsymbol{\mathit{dba}\ell}}

\def\balg{\boldsymbol{\mathit{balg}}}

\def\basic{\boldsymbol{\mathit{basic}}}
\def\ubasic{\boldsymbol{\mathit{ubasic}}}
\def\mbasic{\boldsymbol{\mathit{mbasic}}}

\newcommand\lcbasic{\boldsymbol{\mathit{lcbasic}}}
\newcommand\nbasic{\boldsymbol{\mathit{nbasic}}}
\newcommand\cbasic{\boldsymbol{\mathit{cbasic}}}
\newcommand\lbasic{\boldsymbol{\mathit{lbasic}}}

\def\Set{{\sf Set}}

\newcommand{\func}[1]{\operatorname{#1}}

\def\int{{\sf int}}
\newcommand\cl{{\sf cl}}

\def\Id{\func{Id}}
\newcommand\coz{\func{coz}}

\newcommand\C{{\sf Comp}}

\newcommand\creg{{\sf CReg}}
\newcommand\norm{{\sf Norm}}
\newcommand\KHaus{{\sf KHaus}}
\newcommand\LKHaus{{\sf LKHaus}}
\newcommand\Lind{{\sf Lind}}

\begin{document}

\title[Gelfand-Naimark-Stone duality for normal spaces and insertion theorems]{Gelfand-Naimark-Stone duality for normal spaces and insertion theorems}
\author{G.~Bezhanishvili, P.~J.~Morandi, B.~Olberding}
\date{}


\begin{abstract}
Gelfand-Naimark-Stone duality provides an algebraic counterpart of compact Hausdorff spaces in the form of uniformly complete bounded archimedean $\ell$-algebras. In \cite{BMO18d} we extended this duality to completely regular spaces. In this article we use this extension to characterize normal, Lind\"{e}lof, and locally compact Hausdorff spaces. Our approach gives a different perspective on the classical theorems of Kat\v{e}tov-Tong and Stone-Weierstrass.
\end{abstract}

\maketitle

\section{Introduction}

Gelfand-Naimark-Stone duality provides a dual equivalence between the category $\KHaus$ of compact Hausdorff spaces and  the category $\ubal$ of uniformly complete bounded archimedean $\ell$-algebras, thus providing an algebraic interpretation of $\KHaus$.   In the article \cite{BMO18d}, we lifted this duality from $\KHaus$ to the category $\creg$ of completely regular spaces by identifying the category $\ubal$ with a subcategory of the category of basic extensions of $\ell$-algebras (see Definition~\ref{def: basic}) and finding among basic extensions those that are maximal in an appropriate sense (see Definition~\ref{def: maximal}). Roughly, this lift involves working with the traditional Gelfand-Naimark-Stone duality at one end of the basic extension and a ring-theoretic version of Tarski duality at the other, the latter taking the form of a duality between sets and certain Dedekind complete $\ell$-algebras, which we call basic algebras.  Such maximal basic extensions then are the algebraic counterparts of completely regular spaces.   Basic extensions consist of a monomorphism $A \rightarrow B$ in the category $\bal$ of bounded archimedean $\ell$-algebras, where $B$ is a basic algebra and the image of $A$ is join-meet dense in $B$.  The point of view in \cite{BMO18d} suggests that subcategories of $\creg$ other than $\KHaus$ should be identifiable with appropriate subcategories of the category of maximal basic extensions, and that topological features of classes of completely regular spaces should be reflected in algebraic properties of maximal basic extensions $A \rightarrow B$.  
  
In this article we show that this is indeed the case for the categories of normal spaces, Lindel\"of spaces, and locally compact Hausdorff spaces. For each such category, we axiomatize the corresponding basic extensions and thus are able to give algebraic counterparts for these spaces in the spirit of Gelfand-Naimark-Stone duality.  Our approach to normal spaces depends on well-known insertion theorems of continuous real-valued functions. We show these insertion theorems translate in a natural way to purely algebraic setting of basic extensions and can thus be used to characterize the basic extensions that correspond to normal spaces. In the setting of compact Hausdorff spaces, our approach yields a different perspective on  Gelfand-Naimark-Stone duality and provides an alternative proof of a version of the Stone-Weierstrass theorem.

This article is organized as follows. In Section 2 we recall the necessary background from \cite{BMO18d}. In Section 3 we connect normality and various insertion theorems. In Section 4 we provide characterizations of compact and Lindel\"{o}f spaces. Finally, in Section 5 we characterize locally compact Hausdorff spaces and describe one-point compactifications by means of minimal extensions.
  
\section{Preliminaries}

Gelfand-Naimark-Stone duality, which  
 states that the category of compact Hausdorff spaces is dually equivalent to the category of uniformly
complete bounded archimedean $\ell$-algebras, can be viewed as the restriction of a dual adjunction between the category $\creg$ of completely regular spaces and the category $\bal$ of bounded archimedean $\ell$-algebras.  This  adjunction   does not restrict to an equivalence between $\creg$  and any full subcategory of $\bal$. Thus, in seeking a duality for completely regular spaces that extends Gelfand-Naimark-Stone duality, one must look beyond the category $\bal$. One way to do this is by introducing the notion of a basic extension of a bounded archimedean $\ell$-algebra as done in \cite{BMO18d}. In this section we first describe the adjunction between $\creg$ and $\bal$ and then review the  notions of basic algebras and basic extensions from \cite{BMO18d}.

 To make the paper self-contained, we recall several definitions.
\begin{itemize}
\item A ring $A$ with a partial order $\le$ is an \emph{$\ell$-ring} (that is, a \emph{lattice-ordered ring}) if $(A,\le)$ is a lattice, $a\le b$ implies
$a+c \le b+c$ for each $c$, and $0 \leq a, b$ implies $0 \le ab$.
\item An $\ell$-ring $A$ is \emph{bounded} if for each $a \in A$ there is $n \in \mathbb{N}$ such that $a \le n\cdot 1$ (that is, $1$ is
a \emph{strong order unit}).
\item An $\ell$-ring $A$ is \emph{archimedean} if for each $a,b \in A$, whenever $na \le b$ for each $n \in \mathbb{N}$, then $a \le 0$.
\item An $\ell$-ring $A$ is an \emph{$\ell$-algebra} if it is an $\mathbb R$-algebra and for each $0 \le a\in A$ and $0\le r\in\mathbb R$ we
have $0 \le ra$.
\item Let $\bal$ be the category of bounded archimedean $\ell$-algebras and unital $\ell$-algebra homomorphisms.
\item Let $A\in\bal$. For $a\in A$, define the \emph{absolute value} of $a$ by $|a|=a\vee(-a)$ and the \emph{norm} of $a$ by
$||a||=\inf\{r\in\mathbb R : |a|\le r\}$.\footnote{We view $\mathbb{R}$ as an $\ell$-subalgebra of $A$ by identifying $r$ with $r\cdot 1$.} Then $A$ is \emph{uniformly complete} if the norm is complete.
\item Let $\ubal$ be the full subcategory of $\bal$ consisting of uniformly complete $\ell$-algebras.
\end{itemize}

\subsection{The contravariant functor $C^*:\creg\to\bal$}

For a completely regular space $X$, let $C(X)$ be the ring of continuous real-valued functions, and let $C^*(X)$ be the subring of $C(X)$
consisting of bounded functions. We note that if $X$ is compact, then $C^*(X)=C(X)$. There is a natural partial order $\le$ on $C(X)$
lifted from $\mathbb R$, making $C(X)$ an archimedean $\ell$-algebra. Since $C^*(X)$ is bounded, we then have $C^*(X)\in\bal$. Moreover,
there is a natural norm on $C^*(X)$ given by $||f||=\sup\{|f(x)| : x\in X\}$, which is complete. Thus, $C^*(X)\in\ubal$.

For a continuous map $\varphi:X\to Y$ between completely regular spaces, let $\varphi^* : C^*(Y)\to C^*(X)$ be given by
$\varphi^*(f) = f \circ \varphi$. Then $\varphi^*$ is a unital $\ell$-algebra homomorphism. This yields a contravariant functor
$C^* : \creg \to \bal$ which sends each $X\in\creg$ to the uniformly complete $\ell$-algebra $C^*(X)$, and each continuous map
$\varphi:X\to Y$ to the unital $\ell$-algebra homomorphism $\varphi^* : C^*(Y)\to C^*(X)$. We denote the restriction of $C^*$ to
$\KHaus$ by $C$ since for $X\in\KHaus$ we have $C^*(X)=C(X)$.

\subsection{The contravariant functor $Y:\bal\to\creg$}

For $A\in\bal$, we recall that an ideal $I$ of $A$ is an \emph{$\ell$-ideal} if
$|a|\le|b|$ and $b\in I$ imply $a\in I$, and that $\ell$-ideals are exactly the kernels of $\ell$-algebra homomorphisms. Let $Y_A$ be
the space of maximal $\ell$-ideals of $A$, whose closed sets are exactly sets of the form
\[
Z_\ell(I) = \{M\in Y_A\mid I\subseteq M\},
\]
where $I$ is an $\ell$-ideal of $A$. The space $Y_A$ is often referred to as the \emph{Yosida space} of $A$, and it is well known that
$Y_A\in\KHaus$.

For a unital $\ell$-algebra homomorphism $\alpha:A\to B$ let $\alpha_*:Y_B\to Y_A$ be the continuous map given by $\alpha_*(M)=\alpha^{-1}(M)$.
This defines a contravariant functor $Y:\bal\to\creg$ which sends each $A\in\bal$ to the compact Hausdorff space $Y_A$, and each unital
$\ell$-algebra homomorphism $\alpha:A\to B$ to the continuous map $\alpha_*:Y_B\to Y_A$.

\subsection{Dual adjunction and equivalence}

The functors $C^* : \creg \to \bal$ and $Y: \bal \to \creg$ yield a contravariant adjunction between $\bal$ and $\creg$. To see this, for $A \in \bal$ and a maximal $\ell$-ideal $M$ of $A$, we have a unique unital $\mathbb{R}$-algebra isomorphism $A/M\to\mathbb R$. Therefore, with each $a\in A$,
we can associate $f_a\in C(Y_A)$ given by $f_a(M)=a+M$. Then $\zeta_A:A\to C(Y_A)$ given by $\zeta_A(a)=f_a$ is a unital $\ell$-algebra
homomorphism, which is a monomorphism since the intersection of maximal $\ell$-ideals is 0. 
Since the $\ell$-subalgebra $\zeta_A[A]$ of $C(Y_A)$ separates points of $Y_A$ and $C(Y_A)$ is uniformly complete, $C(Y_A)$ is the uniform completion of $\zeta_A[A]$ by the Stone-Weierstrass theorem. From now on we will view $C(Y_A)$ as the uniform completion $\widehat{A}$ of $A$.
It is then clear that $A$ is uniformly complete iff $\zeta_A$ is an isomorphism.

For $X\in\creg$, associate to each $x\in X$ the maximal $\ell$-ideal
\[
M_x := \{ f \in C^*(X) \mid f(x) = 0 \}.
\]
Then $\xi_X:X\to Y_{C^*(X)}$ given by $\xi_X(x)=M_x$ is a topological embedding, and it is a homeomorphism iff $X$ is compact.

This yields a natural bijection
\[
\mathrm{hom}_{\bal}(A,C^*(X))\simeq\mathrm{hom}_{\creg}(X,Y_A)
\]
which can be described as follows. Let $A\in\bal$ and $X\in\creg$. To each morphism $ \alpha : A \to C^*(X)$ in $\bal$, associate the continuous map $\alpha_* \circ \xi_X : X \to Y_A$. Going the other direction, to each continuous map $\varphi : X \to Y_A$ associate the morphism $\varphi^*\circ \zeta_A : A \to C^*(X)$ in $\bal$. 
This gives a contravariant adjunction $(C^*, Y)$ between $\bal$ and $\creg$. 
The image of $C^*$ is $\ubal$ and the image of $Y$ is $\KHaus$. Consequently, the contravariant adjunction $(C^*,Y)$ between $\creg$ and $\bal$ restricts to a dual equivalence $(C,Y)$ between
$\KHaus$ and $\ubal$, and we arrive at the following celebrated result:

\begin{theorem} [Gelfand-Naimark-Stone duality]
The categories $\KHaus$ and $\ubal$ are dually equivalent, and the dual equivalence is established by the functors $C$ and $Y$.
\label{thm:ubal}
\end{theorem}

\subsection{Dedekind complete algebras and basic algebras}
Our extension of  Gelfand-Naimark-Stone duality to completely regular spaces in \cite{BMO18d} depends on the notion of a basic extension, which in turn depends on that of  a basic algebra. These basic algebras are special cases of {\it Dedekind complete algebras},   
those  $A\in\bal$ with the property that each subset of $A$ bounded above has a least upper bound, and hence each
subset bounded below has a greatest lower bound. Let $\dbal$ be the full subcategory of $\bal$ consisting of Dedekind complete $\ell$-algebras. It is well known (see, e.g., \cite[Rem.~3.5]{BMO13b}) that if $A \in \bal$ is Dedekind complete, then $A$ is uniformly complete, so $\dbal$ is a full subcategory of $\ubal$.

We also recall that a topological space is \emph{extremally disconnected} if the closure of each open set is open.
Let $\sf ED$ be the full subcategory of $\KHaus$ consisting of extremally disconnected spaces. By the Stone-Nakano theorem, for $X\in\KHaus$ we have $C(X)$ is Dedekind complete iff $X\in\sf ED$. This together with Gelfand-Naimark-Stone duality yields:

\begin{theorem} \label{cor:dbal}
The categories $\sf ED$ and $\dbal$ are dually equivalent, and the dual equivalence is established by restricting the functors
$C$ and $Y$.
\end{theorem}

For $A\in\bal$, let $\Id(A)$ be the boolean algebra of idempotents. A nonzero idempotent $e$ is called \emph{primitive} if for
each $f \in \Id(A)$, from $f \leq e$ it follows that $f = 0$ or $f = e$. Thus, primitive idempotents are exactly the atoms of $\Id(A)$.

\begin{definition}
\begin{enumerate}
\item[]
\item We call $A\in\bal$ a \emph{basic algebra} if $A$ is Dedekind complete and $\Id(A)$ is atomic.
\item We call a unital $\ell$-algebra homomorphism $\alpha:A\to B$ between $A,B\in\dbal$ \emph{normal} if $\alpha$ preserves all
existing joins (and hence all existing meets).
\item Let $\balg$ be the category of basic algebras and normal $\ell$-algebra homomorphisms.
\end{enumerate}
\end{definition}

\begin{remark}
The category $\balg$ is not a full subcategory of $\bal$ since not every unital $\ell$-algebra homomorphism between basic algebras is normal.
\end{remark}

As we recall in Theorem~\ref{Tarski}, the category $\balg$ is dually equivalent to the category $\Set$. This suggests an analogy between Tarski duality and the duality of Theorem~\ref{Tarski}. In fact,
the objects in $\balg$ admit a functional representation that resembles in several ways  the representation of complete and atomic Boolean algebras as powersets.

In more detail, for each set $X$, let $B(X)$ be the $\ell$-algebra of bounded real-valued functions on $X$. Then $B(X)\in\balg$. For a function
$\varphi:X\to Y$ between two sets, let $\varphi^+:B(Y)\to B(X)$ be given by $\varphi^+(f)=f\circ\varphi$. Then $\varphi^+$
is a normal $\ell$-algebra homomorphism. This defines a contravariant functor $B:{\sf Set}\to\balg$ from the category $\sf Set$
of sets to the category $\balg$ of basic algebras.

\begin{remark}
If $X,Y \in \creg$ and $\varphi : X \to Y$ is continuous, then $\varphi^*$ is the restriction of $\varphi^+$ to $C^*(Y)$.
\end{remark}

For $A\in\bal$, let $X_A$ be the set of isolated points of $Y_A$. If $A \in \balg$, then $A$ is Dedekind complete, so $Y_A$ is extremally disconnected; and since $\Id(A)$ is atomic, $X_A$ is dense in $Y_A$. For a normal $\ell$-algebra homomorphism $\alpha:A\to B$
between basic algebras, let $\alpha_+:X_B\to X_A$ be the restriction of $\alpha_*:Y_B\to Y_A$ to $X_B$. Then $\alpha_+$ is
a well-defined function, and we have a contravariant functor $X:\balg\to\sf Set$.

\begin{theorem} \cite[Thm.~3.10]{BMO18d} \label{Tarski}
The categories $\sf Set$ and $\balg$ are dually equivalent, and the dual equivalence is established by the functors $B$ and $X$.
\end{theorem}

The isomorphism $\vartheta_A:A \rightarrow B(X_A)$ is defined by $\vartheta_A=\kappa_A\circ\zeta_A$, where
$\zeta_A : A \rightarrow C(Y_A)$ is the Yosida representation and $\kappa_A: C(Y_A) \rightarrow B(X_A)$ sends
$f \in C(Y_A)$ to its restriction to $X_A$.
\[
\begin{tikzcd}
A \arrow[rr, "\vartheta_A"] \arrow[dr, "\zeta_A"'] && B(Y_A) \\
& C(Y_A) \arrow[ru, "\kappa_A"']
\end{tikzcd}
\]
Also, the bijection $\eta_X:X\to X_{B(X)}$ is defined by $\eta_X(x)=\{f\in B(X)\mid f(x)=0\}$.

\begin{remark}
If $X \in \creg$, then $\eta_X \cap C^*(X) = M_x$.
\end{remark}

\subsection{Compactifications and basic extensions}
We now describe our extension of Gelfand-Naimark Stone duality to $\creg$. This duality is fundamental for the rest of the paper, which is devoted to working out how various classes of completely regular spaces can be axiomatized in terms of basic extensions. We first formulate our duality more generally in terms of the category of  compactifications.

We recall (see, e.g., \cite[Sec.~3.5]{Eng89}) that a \emph{compactification} of a completely regular space $X$ is a pair $(Y,e)$,
where $Y$ is a compact Hausdorff space and $e : X\to Y$ is a topological embedding such that the image $e[X]$ is dense in $Y$.
Suppose that $e : X \to Y$ and $e' : X \to Y'$ are compactifications. As usual, we write $e \le e'$ provided there is a continuous map
$f : Y' \to Y$ with $f \circ e' = e$.
\[
\begin{tikzcd}
X \arrow[r, "e'"] \arrow[rd, "e"'] & Y' \arrow[d, "f"] \\
& Y
\end{tikzcd}
\]
The relation $\le$ is reflexive and transitive. Two compactifications $e$ and $e'$ are said to be \emph{equivalent} if $e \le e'$ and $e' \le e$.
It is well known that $e$ and $e'$ are equivalent iff there is a homeomorphism $f : Y'\to Y$ with $f \circ e' = e$. The equivalence classes of
compactifications of $X$ form a poset whose largest element is the Stone-\v{C}ech compactification $s : X \to \beta X$.

\begin{definition}
Let $\C$ be the category whose objects are compactifications $e:X\to Y$ and whose morphisms are pairs $(f,g)$ of continuous maps such that
the following diagram commutes.
\[
\begin{tikzcd}[column sep=5pc]
X \arrow[r, "e"] \arrow[d, "f"'] & Y \arrow[d, "g"] \\
X' \arrow[r, "e'"'] & Y'
\end{tikzcd}
\]
The composition of two morphisms $(f_1,g_1)$ and $(f_2,g_2)$ is defined to be $(f_2\circ f_1, g_2\circ g_1)$.
\[
\begin{tikzcd}[column sep=5pc]
X_1 \arrow[r, "e_1"] \arrow[dd, bend right = 35, "f_2 \circ f_1"'] \arrow[d, "f_1"] & Y_1 \arrow[d, "g_1"']
\arrow[dd, bend left = 35, "g_2 \circ g_1"] \\
X_2 \arrow[r, "e_2"] \arrow[d, "f_2"] & Y_2 \arrow[d, "g_2"'] \\
X_3 \arrow[r, "e_3"'] & Y_3
\end{tikzcd}
\]
\end{definition}

\begin{definition} \label{def: basic}
\begin{enumerate}
\item[]
\item Let $\alpha:A\to B$ be a monomorphism in $\bal$. We say that $\alpha[A]$ is \emph{join-meet dense} in $B$
if each element of $B$ is a join of meets of elements of $\alpha[A]$.
\item Let $A\in\bal$, $B\in\balg$, and $\alpha : A \to B$ be a monomorphism in $\bal$. We call $\alpha : A \to B$ a
\emph{basic extension} if $\alpha[A]$ is join-meet dense\footnote{Equivalently, $\alpha[A]$ is meet-join dense in $B$, meaning that each element of $B$ is a meet of joins of elements of $\alpha[A]$. For details see \cite[Rem.~4.6(2)]{BMO18d}.} in $B$.
\item Let $\basic$ be the category whose objects are basic extensions and whose morphisms are pairs $(\rho, \sigma)$
of morphisms in $\bal$ with $\sigma$ normal and $\sigma\circ\alpha = \alpha' \circ \rho$.
\[
\begin{tikzcd}[column sep=5pc]
A \arrow[r, "\alpha"] \arrow[d, "\rho"'] & B \arrow[d, "\sigma"] \\
A' \arrow[r, "\alpha'"'] & B'
\end{tikzcd}
\]
The composition of two morphisms $(\rho_1, \sigma_1)$ and $(\rho_2, \sigma_2)$ is defined to be $(\rho_2\circ \rho_1, \sigma_2\circ \sigma_1)$.
\[
\begin{tikzcd}[column sep=5pc]
A_1 \arrow[dd, bend right = 35, "\rho_2 \circ \rho_1"'] \arrow[r, "\alpha_1"] \arrow[d, "\rho_1"] & B_1 \arrow[d, "\sigma_1"']
\arrow[dd, bend left=35, "\sigma_2 \circ \sigma_1"] \\
A_2 \arrow[r, "\alpha_2"] \arrow[d, "\rho_2"] & B_2 \arrow[d, "\sigma_2"'] \\
A_3 \arrow[r, "\alpha_3"'] & B_3
\end{tikzcd}
\]
\item Let $\ubasic$ be the full subcategory of $\basic$ consisting of the basic extensions $\alpha : A \to B$ where $A \in \ubal$.
\end{enumerate}
\end{definition}

By \cite[Thm.~4.10]{BMO18d}, $\ubasic$ is a reflective subcategory of $\basic$, and the reflector $r : \basic \to \ubasic$ is defined as follows. Let $\alpha : A \to B$ be a basic extension and $\widehat{A} := C(Y_A)$ the uniform completion of $A$. Then  $\zeta_A : A \to \widehat{A}$ is a monomorphism and $\zeta_B : B \to \widehat{B}$ is an isomorphism. Define $\widehat{\alpha} : \widehat{A} \to B$ to be $\widehat{\alpha} = \zeta_B^{-1} \circ (\alpha_*)^*$. The following diagram is commutative, and setting $r(\alpha) = \widehat{\alpha}$ is the desired reflector.
\[
\begin{tikzcd} [column sep=5pc]
A \arrow[r, "\alpha"] \arrow[d, "\zeta_A"'] & B \\
\widehat{A} \arrow[r, "(\alpha_*)^*"'] & \widehat{B} \arrow[u, "\zeta_B^{-1}"']
\end{tikzcd}
\]

Define a contravariant functor ${\sf E} : \C \to \basic$ as follows. If $e : X \to Y$ is a compactification, let $e^\flat : C(Y) \to B(X)$
be given by $e^\flat(f) = f\circ e$. By \cite[Prop.~4.4]{BMO18d}, $e^\flat$ is a basic extension and we define ${\sf E}(e)=e^\flat$.
For a morphism $(f,g)$ in $\C$
\[
\begin{tikzcd}[column sep = 5pc]
X \arrow[r, "e"] \arrow[d, "f"'] & Y \arrow[d, "g"] \\
X' \arrow[r, "e'"'] & Y'
\end{tikzcd}
\]
define ${\sf E}(f,g)$ to be the pair $(g^*, f^+)$
\[
\begin{tikzcd}[column sep = 5pc]
C(Y') \arrow[ r, "(e')^\flat"] \arrow[d, "g^*"'] & B(X') \arrow[d, "f^+"] \\
C(Y) \arrow[r, "e^\flat"'] & B(X)
\end{tikzcd}
\]

For a morphism $\alpha:A\to B$ in $\bal$, let $\alpha_\flat : X_B \to Y_A$ be the restriction of $\alpha_* : Y_B \to Y_A$ to $X_B$.
Let $\alpha : A \to B$ be a monomorphism in $\bal$ with $B$ a basic algebra. Define a topology $\tau_\alpha$ on $X_B$ as the least topology
making $\alpha_\flat:X_B \to Y_A$ continuous.
By \cite[Thm.~5.5]{BMO18d}, if $\alpha : A \to B$ is a basic extension, then $\alpha_\flat : X_B \to Y_A$ is a compactification.

Define a contravariant functor ${\sf C} : \basic \to \C$ as follows. If $\alpha : A \to B$ is a basic extension, set ${\sf C}(\alpha)$ to be
the compactification $\alpha_\flat : X_B \to Y_A$. For a morphism $(\rho, \sigma)$ in $\basic$
\[
\begin{tikzcd}[column sep = 5pc]
A \arrow[r, "\alpha"] \arrow[d, "\rho"'] & B \arrow[d, "\sigma"] \\
A' \arrow[r, "\alpha'"'] & B'
\end{tikzcd}
\]
define ${\sf C}(\rho, \sigma)$ to be $(\sigma_+, \rho_*)$.
\[
\begin{tikzcd}[column sep = 5pc]
X_{B'} \arrow[r, "\alpha'_\flat"] \arrow[d, "\sigma_+"'] & Y_{A'} \arrow[d, "\rho_*"] \\
X_B \arrow[r, "\alpha_\flat"'] & Y_A
\end{tikzcd}
\]

\begin{theorem} \cite[Thm.~6.3]{BMO18d}\label{duality cptf}
The functors $\sf E:\C \rightarrow \basic$ and $\sf C:\basic \rightarrow \C$ define a dual adjunction of categories that restricts to a
dual equivalence between $\C$ and $\ubasic$.
\end{theorem}

\subsection{Maximal extensions and completely regular spaces}

\begin{definition} \label{def: maximal}
\begin{enumerate}
\item[]
\item We call two basic extensions $\alpha : A \to B$ and $\gamma : C \to B$ \emph{compatible} if the topologies $\tau_\alpha$ and
$\tau_\gamma$ on $X_B$ are equal.
\item A basic extension $\alpha : A \to B$ is \emph{maximal} provided that for every compatible extension $\gamma : C \to B$, there is a
morphism $\delta : C \to A$ in $\bal$ such that $\alpha \circ \delta = \gamma$.
\[
\begin{tikzcd}
A \arrow[rr, "\alpha"] && B \\
& C \arrow[ul, "\delta"] \arrow[ur, "\gamma"']
\end{tikzcd}
\]
\item Let $\mbasic$ be the full subcategory of $\basic$ consisting of maximal basic extensions.
\end{enumerate}
\end{definition}

By \cite[Prop.~7.4]{BMO18d}, a basic extension $\alpha : A \to B$ is maximal iff $A \in \ubal$ and $\alpha_\flat : X_B \to Y_A$ is equivalent to
the Stone-\v{C}ech compactification of $X_B$. Consequently, $\alpha$ is isomorphic to $s^\flat:C(\beta X_B) \to B(X_B)$. The dual equivalence of
Theorem~\ref{duality cptf} then restricts to a dual equivalence between Stone-\v{C}ech compactifications and $\mbasic$. Since the category $\creg$
is equivalent to the subcategory of $\C$ consisting of Stone-\v{C}ech compactifications \cite[Prop.~6.8]{BMO19a}, we obtain the following
generalization of Gelfand-Naimark-Stone duality to completely regular spaces.

\begin{theorem} \cite[Thm.~7.7]{BMO18d} \label{CR duality}
There is a dual equivalence between $\creg$ and $\mbasic$.
\end{theorem}

\section{Normal extensions and insertion theorems}

In this section we develop algebraic counterparts of normal spaces by introducing normal basic extensions. Our results utilize 
well-known insertion theorems from the literature, which we formulate in the language of basic extensions.

\begin{definition}
Let $\alpha : A \to B$ be a basic extension.
\begin{enumerate}
\item $b \in B$ is \emph{closed} (relative to $\alpha$) if $b$ is a meet of elements from $\alpha[A]$.
\item $b \in B$ is \emph{open} (relative to $\alpha$) if $b$ is a join of elements from $\alpha[A]$.
\item $b \in B$ is \emph{clopen} (relative to $\alpha$) if $b$ is both closed and open.
\end{enumerate}
\end{definition}

\begin{remark} \label{rem: closed}
Let $\alpha : A \to B$ be a basic extension. The following facts are easy to see.
\begin{enumerate}
\item The meet of an arbitrary set of closed elements is closed, and the join of an arbitrary set of open elements is open. In addition, the join of finitely many closed elements is closed, and the meet of finitely many open elements is open.
\item The sum of closed elements is closed, and the sum of open elements is open. These follow from the $\ell$-ring identities
\[
\bigwedge S + \bigwedge T  = \bigwedge \{ s+t \mid s \in S, t \in T\} \quad\textrm{and}\quad \bigvee S + \bigvee T = \bigvee \{s + t \mid s \in S, t \in T\}.
\]
\item $b\in B$ is closed (resp.\ open) iff $-b$ is open (resp.\ closed). These follow from the $\ell$-ring identities
\[
-\bigwedge S = \bigvee \{ -s \mid s \in S\} \quad\textrm{and}\quad -\bigvee S = \bigwedge \{ -s \mid s \in S \}.
\]
\item If $a \in A$ and $b\in B$ is open (resp.~closed), then $\alpha(a) \wedge b$ is open (resp.~$\alpha(a) \vee b$ is closed). These follow from the distributive laws
\[
s \wedge \bigvee T = \bigvee \{ s \wedge t \mid t \in T\} \quad\textrm{and}\quad  s \vee \bigwedge T  = \bigwedge \{ s \vee t \mid t \in T\}.
\]
\item Each element of $B$ is a meet of open elements and a join of closed elements. 
\end{enumerate}
\end{remark}

Let $X$ be a completely regular space and let $f \in B(X)$. For $x \in X$ let $\mathcal{N}_x$ be the set of all open neighborhoods of $x$. 
Define $f^*, f_* \in B(X)$ by
\begin{align*}
f^*(x) &= \inf \{ \sup f(U) \mid U \in \mathcal{N}_x \}, \\
f_*(x) &= \sup \{ \inf f(U) \mid U \in \mathcal{N}_x \}.
\end{align*}

Then $f$ is \emph{upper semicontinuous} if $f^* = f$, and \emph{lower semicontinuous} if $f_* = f$. It is pointed out in 
\cite[p.~430]{Dil50} that $f$ is upper semicontinuous iff $f^{-1}(-\infty,r)$ is open, and $f$ is lower semicontinous iff $f^{-1}(r, \infty)$ 
is open, for each $r \in \mathbb{R}$. It follows that $f$ is continuous iff it is both upper and lower semicontinous.

\begin{lemma} \label{lem: semicontinuous}
Let $e : X \to Y$ be a compactification and $e^\flat : C(Y) \to B(X)$ the corresponding basic extension.
\begin{enumerate}
\item $f \in B(X)$ is closed $($relative to $e^\flat)$ iff $f$ is upper semicontinuous.
\item $f \in B(X)$ is open $($relative to $e^\flat)$ iff $f$ is lower semicontinuous.
\item $f \in B(X)$ is clopen $($relative to $e^\flat)$ iff $f$ is continuous.
\end{enumerate}
\end{lemma}

\begin{proof}
(1). Let $\iota : C^*(X) \to B(X)$ be the inclusion. Then $e^\flat$ is the composition $\iota \circ e^* : C(Y) \to C^*(X) \to B(X)$. If $f$ is a meet from $e^\flat[C(Y)]$,
then it is a meet from $\iota[C^*(X)]$, and so $f$ is upper semicontinuous (see, e.g., \cite[Lem.~4.1]{Dil50}). Conversely,
suppose that $f$ is upper semicontinuous.
By \cite[Lem.~7.2]{BMO18d}, there is an upper semicontinuous function $f' \in B(Y)$ with $e^+(f') = f$. By \cite[Lem.~4.1]{Dil50},
$f' = \bigwedge S$ for some $S \subseteq C(Y)$. Because $e^+$ is a complete homomorphism, we have
\[
f = e^+(f') = e^+\left(\bigwedge S\right) = \bigwedge \{ e^+(s) \mid s \in S \} = \bigwedge \{ e^*(s) \mid s \in S \} = \bigwedge \{ e^\flat(s) \mid s \in S\},
\]
which shows that $f$ is closed relative to $e^\flat$.

(2). If $f \in B(X)$, then $f$ is lower semicontinuous iff $-f$ is upper semicontinuous. By Remark~\ref{rem: closed}(3), $f$ is a join from $e^\flat[C(Y)]$ iff $-f$
is a meet from $e^\flat[C(Y)]$. Thus, (2) follows from (1).

(3). A real-valued function is continuous iff it is both upper and lower semicontinuous. Now apply (1) and (2).
\end{proof}

\begin{lemma} \label{lem: closed}
Let $e : X \to Y$ be a compactification, $e^\flat : C(Y) \to B(X)$ the corresponding basic extension, and $S\subseteq X$.
\begin{enumerate}
\item $S$ is a closed subset of $X$ iff $\chi_S$ is closed $($relative to $e^\flat)$.
\item $S$ is an open subset of $X$ iff $\chi_S$ is open $($relative to $e^\flat)$.
\item $S$ is a clopen subset of $X$ iff $\chi_S$ is clopen $($relative to $e^\flat)$.
\end{enumerate}
\end{lemma}

\begin{proof}
(1). It is elementary to see that $S$ is closed in $X$ iff $\chi_S$ is upper semicontinuous. Thus, by Lemma~\ref{lem: semicontinuous}(1),
$S$ is closed iff $\chi_S$ is closed relative to $e^\flat$.

(2). We have that $S$ is open in $X$ iff $\chi_S$ is lower semicontinuous. Therefore, by Lemma~\ref{lem: semicontinuous}(2), $S$ is open iff
$\chi_S$ is open.

(3). This follows from (1) and (2).
\end{proof}

 Normal spaces can be characterized by properties that involve insertion of continuous real-valued functions between functions from more general classes. As discussed in Remark~\ref{KT remark}, such theorems originate in the early work of Baire and Hahn 
 and take one of their most well-known forms in the work of Kat\v{e}tov \cite{Kat51, Kat53} and Tong \cite{Ton52}: A space $X$ is normal iff for each pair of bounded real-valued functions $f \leq g$ on $X$ with $f$ upper semicontinuous and $g$ lower semicontinuous, there is a continuous real-valued function $h$ on $X$ such that $f \leq h \leq g$. Tong's approach to insertion emphasizes lattice-theoretic conditions and hence lends itself to our algebraic approach.  In particular, the strategy of Tong is to insert first between $f$ and $g$ a pair $f_1 \leq g_1$ where $f_1$ is a countable meet of continuous functions and $g_1$ is a countable join of continuous functions. He then shows that $f_1$ and $g_1$ can be replaced by a single function $h$ that is a countable meet and a countable join of continuous functions  (see Theorem~\ref{Tong} below). Since a meet of upper semicontinuous functions is upper semicontinuous and a join of lower semicontinuous functions is lower semicontinuous, it follows that $h$ is continuous. In property (T) of  the following definition we extract the first step in Tong's strategy. We consider also the closely related condition (BS), which was studied by Blatter and Seever in the function setting in \cite{BS75}. Finally, condition (S) is a common strengthening of both conditions.

\begin{definition}
Let $\alpha : A \to B$ be a basic extension, $f \in B $ closed, and $g \in B$ open.
\begin{enumerate}
\item[(T)] If $f \le g$, then there are $\{a_n\},\{b_n\}\subseteq A$ with $f \le \bigwedge_n \alpha(a_n) \le \bigvee_n \alpha(b_n) \le g$.
\item[(BS)] If $f \le g$, then there are $\{a_n\},\{b_n\}\subseteq A$ with $f \le \bigvee_n \alpha(a_n) \le \bigwedge_n \alpha(b_n) \le g$.
\item[(S)] If $f \le g$, then there are $\{a_n\},\{b_n\}\subseteq A$ with $f \le \bigwedge_n \alpha(a_n) = \bigvee_n \alpha(b_n) \le g$.
\end{enumerate}
\end{definition}

We show that the three conditions are equivalent. For this we utilize Tong's theorem \cite[Thm.~1]{Ton52}, but rephrase it in our setting. This theorem is step two of Tong's strategy described above.

\begin{theorem}[Tong] \label{Tong}
Let $\alpha : A \to B$ be a basic extension, $f$ a countable meet of elements from $\alpha[A]$, and $g$ a countable join of elements from
$\alpha[A]$. If $f \le g$, then there is $u \in B$ such that $f \le u \le g$ and $u$ is both a countable meet and a countable join of elements
from $\alpha[A]$.
\end{theorem}

\begin{proof}
Write $f = \bigwedge_n \alpha(a_n)$ and $g = \bigvee_n \alpha(b_n)$ with the $a_n, b_n \in A$. By replacing $a_n$ by
$a_1 \wedge \cdots \wedge a_n$ and $b_n$ by $b_1 \vee \cdots \vee b_n$, we may assume the $a_n$ are decreasing and the $b_n$ are increasing. Let
\[
u_n = (\alpha(a_1) \wedge \alpha(b_1)) \vee \cdots \vee (\alpha(a_n) \wedge \alpha(b_n)).
\]
Then $u_n \le \alpha(b_n)$ as the $b_i$ are increasing, so $u_n \le g$ for each $n$. Thus, the join $u = \bigvee_n u_n$ exists in $B$,
and $u \le g$. Let $v_n = u_n \vee \alpha(a_n)$. Then $\alpha(a_n) \le v_n$ for each $n$, so $f \le v_n$ for each $n$. The meet
$v = \bigwedge_n v_n$ then exists in $B$. As $u_n, v_n \in \alpha[A]$, it is sufficient to show $u=v$. We do this by showing that
$f \le u \le v$ and $v \le u\vee f$.

We have
\[
v = \bigwedge_n v_n = \bigwedge_n (u_n \vee \alpha(a_n)) \le \bigwedge_n (u \vee \alpha(a_n)) = u \vee \bigwedge \alpha(a_n) = u\vee f.
\]
Next, since the $a_n$ are decreasing, for each $n,p$ we have
\begin{align*}
u_{n+p} &= u_n \vee (\alpha(a_{n+1})\wedge \alpha(b_{n+1})) \vee \cdots \vee (\alpha(a_{n+p}) \wedge \alpha(b_{n+p})) \\
&\le u_n \vee \alpha(a_{n+1}) \vee \cdots \vee \alpha(a_{n+p}) \le u_n \vee \alpha(a_{n}) \vee \cdots \vee \alpha(a_{n+p}) \\
&= u_n \vee \alpha(a_n) = v_n
\end{align*}
This implies that $\bigvee_p u_{n+p} \le v_n$ for each $n$, but $u = \bigvee_p u_{n+p}$ since the $u_n$ are increasing, hence $u \le v_n$
for each $n$. Because $v = \bigwedge_n v_n$, we have $u \le v$. Finally, since $f\le g$,
\[
f = f \wedge g =  f \wedge \bigvee_n \alpha(b_n) = \bigvee_n (f \wedge \alpha(b_n)) \le \bigvee_n (\alpha(a_n) \wedge \alpha(b_n)) = \bigvee_n u_n = u.
\]
This finishes the argument that $u=v$, and so the proof is complete.
\end{proof}

\begin{theorem} \label{thm: BS=T}
For a basic extension $\alpha : A \to B$, \emph{(T)}, \emph{(BS)}, and \emph{(S)} are equivalent.
\end{theorem}

\begin{proof}
(T) $\Rightarrow$ (S). This is Theorem~\ref{Tong}.

(S) $\Rightarrow$ (BS). This is trivial.

(BS) $\Rightarrow$ (T). Suppose that $f\in B$ is closed, $g\in B$ is open, and $f \le g$. By (BS), there are $\{a_n\},\{b_n\}\subseteq A$ with
$f \le \bigvee_n \alpha(a_n) \le \bigwedge_n \alpha(b_n) \le g$. If $f' = \bigwedge_n \alpha(b_n)$, then $f'$ is closed and
$f' \le g$. Applying (BS) to $f' \le g$ yields $\{c_n\}, \{d_n\} \subseteq A$ with
$f' \le \bigvee_n \alpha(c_n) \le \bigwedge_n \alpha(d_n) \le g$. Thus,
$f \le \bigwedge_n \alpha(b_n) \le \bigvee_n \alpha(c_n) \le g$, yielding (T).
\end{proof}

We next connect normality with the  equivalent conditions (T), (BS), and (S).

\begin{theorem} \label{thm: T --> normal}
Let $e : X \to Y$ be a compactification and $e^\flat : C(Y) \to B(X)$ the corresponding basic extension.
\begin{enumerate}
\item If $e^\flat$ satisfies the equivalent conditions \emph{(T)}, \emph{(BS)}, and \emph{(S)}, then $X$ is normal.
\item If $X$ is normal, then $s^\flat : C(\beta X) \to B(X)$ satisfies \emph{(T)}, \emph{(BS)}, and \emph{(S)}.
\item $X$ is normal iff $s^\flat$ satisfies \emph{(T)}, \emph{(BS)}, \emph{(S)}.
\end{enumerate}
\end{theorem}

\begin{proof}
(1). Suppose that $e^\flat$ satisfies (BS) and let $C \subseteq U \subseteq X$ with $C$ closed and $U$ open. Then $\chi_C \le \chi_U$, and by
Lemma~\ref{lem: closed}, $\chi_C$ is closed and $\chi_U$ is open. Since $e^\flat$ satisfies (BS), there are $a_n,b_n \in C(Y)$ with
$\chi_C \le \bigvee_n e^\flat(a_n) \le \bigwedge_n e^\flat(b_n) \le \chi_U$. Because each $e^\flat(b_n)$ is a continuous function on
$X$, the meet $f:=\bigwedge_n e^\flat(b_n)$ is upper semicontinuous. Similarly, the join $g:=\bigvee_n e^\flat(a_n)$ is lower
semicontinuous. Since $g \le f$,
\[
C = \chi_C^{-1}(1/2,\infty) \subseteq g^{-1}(1/2,\infty) \subseteq f^{-1}(1/2,\infty) \subseteq f^{-1}[1/2, \infty) \subseteq
\chi_U^{-1}[1/2, \infty) = U.
\]
Let $V = g^{-1}(1/2,\infty)$ and $D = f^{-1}[1/2, \infty)$. Then $C \subseteq V \subseteq D \subseteq U$. Since $g$ is
lower semicontinuous, $V$ is open, and since $f$ is upper semicontinuous, $D$ closed. Thus, $X$ is normal.

(2). Suppose that $X$ is normal. We show that $s^\flat : C(\beta X) \to B(X)$ satisfies (T) by utilizing the proof strategy of \cite[Thm.~2]{Ton52}.  Let $f,g \in B(X)$ with $f$ closed, $g$ open,
and $f \le g$. By Lemma~\ref{lem: semicontinuous},
$f$ is upper semicontinuous and $g$ is lower semicontinuous. We scale $f,g$ to assume $0 \le f \le g \le 1$ as follows. If $a = \inf(f(X))$, then $0 \le f+a \le g+a$. If $b = \sup((g+a)(X))$, then
\[
0 \le \frac{(f+a)}{b} \le \frac{(g+a)}{b} \le 1.
\]
Replacing $f$ and $g$ by $(f+a)/b$ and $(g+a)/b$, we may assume
$0 \le f \le g \le 1$. We next utilize Urysohn's lemma to produce a countable join of continous functions in between $f$ and $g$. Let $r,s$ be rational numbers with $0 \le r < s \le 1$. From $f\le g$ it follows that
$f^{-1}[s,1]  \subseteq g^{-1}(r,1]$. Since $f$ is upper semicontinuous, $f^{-1}[s,1] = f^{-1}[s,\infty)$ is closed, and since $g$ is lower semicontinuous, $g^{-1}(r,1] = g^{-1}(r,\infty)$ is open. Therefore, by Urysohn's lemma, for each $r,s$
there is $c_{rs} \in C^*(X)$ with $0 \le c_{rs} \le r$ such that $c_{rs} = 0$ on $X \setminus g^{-1}(r,1]$ and $c_{rs} = r$ on $f^{-1}[s,1]$.
From this it is easy to see that $c_{rs} \le g$ for each $r,s$. Since for each $c_{rs}$ there is a unique $d_{rs} \in C(\beta X)$ with $s^\flat(d_{rs}) = c_{rs}$, we have $\bigvee  s^\flat(d_{rs}) \le g$. We next show that $f \le \bigvee  s^\flat(d_{rs})$. Let $x \in X$. Without loss of generality, we may assume that $f(x) > 0$. Let $r < s$ be rationals with $0 \le r < s \le f(x)$. Then $x \in f^{-1}[s,1]$, so $c_{rs}(x) = r$. As $f(x)$ can be approximated
from below by rationals, we see that $f \le \bigvee s^\flat(d_{rs})$. Thus, $f \le \bigvee s^\flat(d_{rs}) \le g$.

Set $g' = \bigvee s^\flat(d_{rs})$. Then $f \le g'$, so $-g' \le -f$. Moreover, $-g'$ is upper semicontinuous and $-f$ is lower semicontinuous.
Repeating the previous argument yields $e_{rs} \in C(\beta X)$ such that $-g' \le \bigvee s^\flat(e_{rs}) \le -f$. Then
$f \le \bigwedge s^\flat(-e_{rs}) \le \bigvee s^\flat(d_{rs}) \le g$. This shows $s^\flat : C(\beta X) \to B(X)$ satisfies (T).

(3). This follows from (1) and (2).
\end{proof}

On the other hand, the converse of Theorem~\ref{thm: T --> normal}(1) is false, as we next show.

\begin{example}
Let $X$ be an uncountable discrete space and let $e : X \to Y$ be its one-point compactification, where $Y=X\cup\{\omega\}$. Clearly $X$ is normal. We show that $e^\flat : C(Y) \to B(X)$ does not satisfy (S). Let $A$ be a subset of $X$ such that both $A$ and $X \setminus A$ are uncountable, and let $f = g = \chi_A$. Since $X$ is discrete, $f$ is closed and $g$ is open relative to $e^\flat$ (see Lemma~\ref{lem: closed}). We show that there do not exist countable families $\{a_n\}, \{b_n\} \subseteq C(Y)$ with $f \le \bigwedge_n e^\flat(a_n) \le \bigvee_n e^\flat(b_n) \le g$. If such families existed, then $\chi_A = \bigwedge_n e^\flat(a_n)$. To see that this cannot happen, we first observe that $a_n(\omega) \ge 1$ for each $n$. Otherwise, there is $n$ with $a_n(\omega) < 1$. So $a_n^{-1}(-\infty, 1-\varepsilon)$ is an open neighborhood of $\omega$ for some $\varepsilon > 0$, which implies that $a_n^{-1}(-\infty, 1-\varepsilon)$ contains all but finitely many elements of $X$. This is impossible from the inequality $\chi_A \le e^\flat(a_n)$ since $A$ is infinite. Consequently, $a_n(\omega) \ge 1$, so $a_n^{-1}(1-\varepsilon, \infty)$ contains all but finitely many elements of $X$ for each $\varepsilon > 0$. This implies that $a_n^{-1}(-\infty, 1/m)$ is finite for all positive integers $m$. From $\chi_A = \bigwedge_n e^\flat(a_n)$ it follows that $X\setminus A \subseteq \bigcup \{ a_n^{-1}(-\infty, 1/m) \mid n,m > 0\}$. This union is then countable, which is a contradiction since $X\setminus A$ is uncountable. Thus, $e^\flat$ does not satisfy (S).
\end{example}

\begin{definition}
\begin{enumerate}
\item[]
\item We call a basic extension $\alpha : A \to B$ \emph{normal} provided $\alpha$ is maximal and satisfies the equivalent conditions (T), (BS),
and (S).
\item Let $\nbasic$ be the full subcategory of $\mbasic$ consisting of normal extensions.
\item Let $\norm$ be the full subcategory of $\creg$ consisting of normal spaces.
\end{enumerate}
\end{definition}

Theorems~\ref{CR duality} and~\ref{thm: T --> normal}(3) yield the following duality theorem for normal spaces.

\begin{theorem}\label{thm: norm}
$\norm$ is dually equivalent to $\nbasic$.
\end{theorem}

Another consequence of 
Theorems~\ref{CR duality} and~\ref{thm: T --> normal}(3) is the following.

\begin{corollary} \label{cor: normal extension} \label{cor: normal}
\begin{enumerate}
\item[]
\item A basic extension $\alpha:A\to B$ is normal iff $X_B$ is a normal space and $\alpha_\flat:X_B\to Y_A$ is the
Stone-\v{C}ech compactification.
\item A completely regular space $X$ is normal iff $\iota : C^*(X) \to B(X)$ is normal.
\end{enumerate}
\end{corollary}

\begin{proof}
(1) is immediate from 
Theorems~\ref{CR duality} and~\ref{thm: T --> normal}(3), and
(2) follows from Theorem~\ref{thm: T --> normal}(3) since $s^\flat$ and $\iota$ are isomorphic basic extensions (see \cite[Example~4.8(2)]{BMO18d}).
\[
\begin{tikzcd}[column sep=5pc]
C(\beta X) \arrow[r, "s^\flat"]  \arrow[d, "s^*"'] & B(X) \arrow[d, equal] \\
 C^*(X) \arrow[r, "\iota"'] & B(X)
\end{tikzcd}
\]
\end{proof}

From this we obtain the following well-known insertion theorem \cite{Kat51, Kat53, Ton52}.

\begin{theorem}[Kat\v{e}tov-Tong] \label{KT}
A $T_1$-space $X$ is normal iff for each $f,g \in B(X)$ with $f$ upper semicontinuous and $g$ lower semicontinuous, $f \le g$ implies that there is
$c \in C^*(X)$ with $f \le c \le g$.
\end{theorem}

\begin{proof}
The proof of the right-to-left direction is standard. Suppose that $A,B$ are disjoint closed sets. Let $W = X \setminus B$. Then $A \subseteq W$ and $W$ is open. Therefore, $\chi_A$ is upper semicontinuous, $\chi_W$ is lower semicontinuous, and $\chi_A \le \chi_W$. The Kat\v{e}tov-Tong condition yields $c \in C^*(X)$ with $\chi_A \le c \le \chi_W$.  Set $U = c^{-1}(1/3, \infty)$ and $V = X \setminus c^{-1}[2/3, 1]$. Then it is straightforward to see that $A \subseteq U$, $B \subseteq V$, and $U \cap V = \varnothing$. Thus, $X$ is normal.

For the left-to-right direction, if $X$ is normal, by Corollary~\ref{cor: normal}(2),
there are countable families $\{a_n\}, \{b_n\} \subseteq C(\beta X)$ with $f \le \bigwedge s^\flat(a_n) = \bigvee s^\flat(b_n) \le g$.
The function $c = \bigwedge s^\flat(a_n) = \bigvee s^\flat(b_n)$ is both upper and lower semicontinouous on $X$ by
Lemma~\ref{lem: semicontinuous}. Thus, $c \in C^*(X)$ and $f \le c \le g$.
\end{proof}

A basic extension $\alpha : A \to B$ is normal if it is both maximal and satisfies (S).
We show that these conditions can be replaced by a single condition.

\begin{definition} \label{U2}
Let $\alpha : A \to B$ be a basic extension, $f\in B$ closed, and $g\in B$ open.
\begin{enumerate}
\item[(N)] If $f \le g$, then there is $a \in \widehat{A}$ with $f \le \widehat{\alpha}(a) \le g$.
\end{enumerate}
\end{definition}

\begin{remark} \label{rem: 3.15} 
For a basic extension $\alpha : A \to B$, if $A$ is uniformly complete, then $\alpha : A \to B$ satisfies (N) iff $f$ closed, $g$ open, and $f \le g$ imply there is $a \in A$ with
$f \le \alpha(a) \le g$.
\end{remark}

\begin{lemma} \label{lem: alpha versus alpha hat}
Let $\alpha : A \to B$ be a basic extension.
\begin{enumerate}
\item If $a \in \widehat{A}$, then there is an increasing \emph{(}resp.~decreasing\emph{)} sequence $\{a_n\}$ from $A$ such that
$a = \lim_{n\to\infty} \zeta_A(a_n)$. 
\item If $a \in \widehat{A}$, then $\widehat{\alpha}(a)$ is a countable join and a countable meet from $\alpha[A]$. 

\item If $f \in B$, then $f$ is closed with respect to $\alpha$ iff $f$ is closed with respect to $\widehat{\alpha}$. 

\item If $g \in B$, then $g$ is open with respect to $\alpha$ iff $g$ is open with respect to $\widehat{\alpha}$. 
\end{enumerate}
\end{lemma}

\begin{proof}
(1). Since $\widehat{A} = C(Y_A)$ is the uniform completion of $A$, there is a
sequence $\{c_n\}$ in $A$ with $\lim_{n\to\infty} \zeta_A(c_n) = a$. This means $\| a- \zeta_A(c_n)\| \to 0$. If
$r_n = \|a- \zeta_A(c_n)\|$, then $-r_n \le a- \zeta_A(c_n) \le r_n$, so $\zeta_A(c_n) - r_n \le a$. Setting
$b_n = c_n - r_n$, we have $\zeta_A(b_n) \le a$ and $\lim_{n\to \infty} \zeta_A(b_n) = a$. Let
$a_n = b_1 \vee \cdots \vee b_n$. Then $\{a_n\}$ is an increasing sequence with $\{\zeta_A(a_n) \}$ bounded by $a$.
Because $\zeta_A(b_n) \le \zeta_A(a_n) \le a$ for each $n$, it follows that $\lim_{n\to \infty} \zeta_A(a_n) = a$.
The proof for finding a decreasing sequence converging to $a$ is dual.

(2). 
By (1), there is an increasing sequence $\{a_n\}$ from $A$ with $a = \lim_{n\to\infty} \zeta_A(a_n)$. 
Since morphisms in $\bal$ are continuous with respect to the norm topology \cite[p.~444]{BMO13a}, 
$\widehat{\alpha}(a) = \lim_{n\to\infty}\widehat{\alpha}(\zeta_A(a_n)) = \lim_{n\to\infty} \alpha(a_n)$. As 
$\alpha(a_n) \le \widehat{\alpha}(a)$ for each $n$, we see that $\bigvee_n \alpha(a_n) \le \widehat{\alpha}(a)$. 
If $b \in B$ with $\alpha(a_n) \le b$ for each $n$, then since $\le$ is a closed relation on $B$, we have 
$\lim_{n\to\infty} \alpha(a_n) \le b$. Thus, $\widehat{\alpha}(a) = \bigvee_n \alpha(a_n)$, a countable join from 
$\alpha[A]$. A similar argument shows that $\widehat{\alpha}(a)$ is also a countable meet from $\alpha[A]$.

(3). First suppose that $f$ is closed with with respect to $\alpha$. Then $f$ is a meet from $\alpha[A]$. 
Because $\alpha[A] \subseteq \widehat{\alpha}\left[\widehat{A}\right]$, it follows that $f$ is a meet from 
$\widehat{\alpha}\left[\widehat{A}\right]$. Therefore, $f$ is closed with respect to $\widehat{\alpha}$. Conversely, suppose 
$f$ is closed with respect to $\widehat{\alpha}$. Then $f = \bigwedge \widehat{\alpha}[T]$ for some $T \subseteq \widehat{A}$. 
By (2), each $\widehat{\alpha}(a) \in \widehat{\alpha}[T]$ is a meet from $\alpha[A]$. Thus, $f$ is a meet from $\alpha[A]$, 
and hence $f$ is closed with respect to $\alpha$.

(4). Since $g$ is open 
iff $-g$ is closed by Remark~\ref{rem: closed}(3), this follows from (3) applied to $f = -g$.
\end{proof}

\begin{remark}
For a basic extension $\alpha : A \to B$, we cannot replace $\widehat{A}$ by $A$ in the definition of (N). To see this, suppose that $A$ is not uniformly complete,
so there is $a\in \widehat{A}\setminus \zeta_A[A]$. Let $f = g = \widehat{\alpha}(a)$. 
Then $f,g$ are closed and open relative to $\alpha$ by (3) and (4) of Lemma~\ref{lem: alpha versus alpha hat}. If there is $c \in A$ with 
$f \le \alpha(c) \le g$, then $a=\zeta_A(c)$, a contradiction.
\end{remark}

\begin{proposition}
Let $\alpha : A \to B$ be a basic extension.
\begin{enumerate}
\item $\alpha : A \to B$ satisfies \emph{(N)} iff $\widehat{\alpha} : \widehat{A} \to B$ satisfies \emph{(N)}.

\item $\alpha$ satisfies either of \emph{(T)}, \emph{(BS)}, \emph{(S)} iff so does $\widehat{\alpha}$. 
\end{enumerate}
\end{proposition}

\begin{proof}
(1). This follows from Remark~\ref{rem: 3.15} and (3) and (4) of Lemma~\ref{lem: alpha versus alpha hat}.

(2). Suppose $\alpha$ satisfies (S). Let $f \le g$ with $f \in B$ closed and $g\in B$ open relative to $\widehat{\alpha}$. 
By (3) and (4) of Lemma~\ref{lem: alpha versus alpha hat}, $f$ is closed and $g$ is open relative to $\alpha$. Since $\alpha$ satisfies (S), there are $a_n, b_n \in A$ with 
$f \le \bigwedge_n \alpha(a_n) = \bigvee_n \alpha(b_n) \le g$. Therefore, 
$f \le \bigwedge_n \widehat{\alpha}(\zeta_A(a_n)) = \bigvee_n \widehat{\alpha}(\zeta_A(b_n)) \le g$, so $\widehat{\alpha}$ satisfies (S). 
Conversely, suppose that $\widehat{\alpha}$ satisfies (S). Let $f\le g$ with $f$ closed and $g$ open relative to $\alpha$. Using (3) and (4) of Lemma~\ref{lem: alpha versus alpha hat} again, 
$f$ is closed and $g$ is open relative to $\widehat{\alpha}$, so there are $c_n, d_n \in \widehat{A}$ with 
$f \le \bigwedge_n \widehat{\alpha}(c_n) = \bigvee_n \widehat{\alpha}(d_n) \le g$. By Lemma~\ref{lem: alpha versus alpha hat}(2), there are $a_{nm}, b_{nm} \in A$ with 
$\widehat{\alpha}(c_n) = \bigwedge_m \alpha(a_{nm})$ and $\widehat{\alpha}(d_n) = \bigvee_m \alpha(b_{nm})$.
Thus, $f \le \bigwedge_{n,m} \alpha(a_{nm}) = \bigvee_{n,m} \alpha(b_{nm}) \le g$. Consequently, $\alpha$ satisfies (S).

\end{proof}

\begin{theorem} \label{U2 implies normal}
Let $e : X \to Y$ be a compactification and $e^\flat : C(Y) \to B(X)$ the corresponding basic extension.
Then $e^\flat$ satisfies \emph{(N)} iff $e^\flat$ is normal.
\end{theorem}

\begin{proof}
Suppose that $e^\flat$ satisfies (N). To show that $e^\flat$ is normal, by Theorem~\ref{thm: T --> normal} it is sufficient to show that $X$ is normal and $e$ is isomorphic to the Stone-\v{C}ech compactification. For this, by
\cite[Cor.~3.6.4]{Eng89} it suffices to show that if $C, D$ are disjoint closed sets in $X$, then
$\cl_Y(e(C)) \cap \cl_Y(e(D)) = \varnothing$. Let $U = X \setminus D$. Then $C \subseteq U$, so $\chi_C \le \chi_U$.
By Lemma~\ref{lem: closed}, $\chi_C$ is closed and $\chi_U$ is open relative to $e^\flat$. As $C(Y)$ is uniformly complete,
by (N), there is $c \in C(Y)$ with $\chi_C \le e^\flat(c) \le \chi_U$. Therefore,
$C \subseteq e^{-1}c^{-1}[1/3, \infty) \subseteq e^{-1}c^{-1}(2/3, \infty) \subseteq U$. We have thus found a closed set
$C':=c^{-1}[1/3, \infty)$ of $Y$ and an open set $U':=c^{-1}(2/3, \infty)$ of $Y$ with $e(C) \subseteq C' \subseteq U'$
and $e^{-1}(U') \subseteq U$. Set $D' = Y \setminus U'$. Then $e(D) \subseteq D'$ and $C' \cap D' = \varnothing$.
Consequently, $\cl_Y(e(C)) \cap \cl_Y(e(D)) = \varnothing$, yielding that $X$ is normal and $e$ is isomorphic to $s$.

For the converse, by Theorem~\ref{thm: T --> normal} it suffices to assume that $e=s$ and $s$ satisfies (S). Let $f\in B(X)$ be closed, $g\in B(X)$ open, and
$f\le g$. By (S), there are $\{a_n\},\{b_n\}\subseteq C(Y)$ such that $f\le \bigwedge s^\flat(a_n)=\bigvee s^\flat(b)\le g$.
Set $c=\bigwedge s^\flat(a_n)=\bigvee s^\flat(b)$. By Lemma~\ref{lem: semicontinuous}, $c$ is both upper and lower semicontinuous
on $X$, hence $c\in C^*(X)$.
But then there is $a \in C(\beta X)$ with $c = s^\flat(a)$. Thus, $s^\flat$ satisfies (N).
\end{proof}

As follows from Theorem~\ref{U2 implies normal}, if $e^\flat : C(Y) \to B(X)$ satisfies (N), then $e^\flat$ is maximal. We next show that 
$e^\flat$ satisfying the equivalent conditions (T), (BS), and (S) does not imply that $e^\flat$ is maximal, so (N) is strictly stronger than 
(T), (BS), and (S). 

\begin{example} \label{ex: (T) weaker than (N)}
Let $X$ be the set of natural numbers with the discrete topology and let $e:X \to Y$ be the one-point compactification. Clearly
$e : X \to Y$ is not isomorphic to the Stone-\v{C}ech compactification, so
$e^\flat$ is not a maximal extension (see Section 2.6). We show that $e^\flat$ satisfies (T). Let $f \in B(X)$. Since $X$ is discrete, $f$ is continuous,
and so by Lemma~\ref{lem: semicontinuous}, $f$ is closed. Thus, there is $S \subseteq C(Y)$ with $f = \bigwedge e^\flat[S]$.
We show that there is a countable subset $S'$ of $S$ with $f = \bigwedge e^\flat[S']$. For each $n,m \in \mathbb{N}$ there is $c_{nm} \in S$
with $c_{nm}(n) \le f(n) + 1/m$. From this it follows that $f = \bigwedge_{n,m} c_{nm}$.
Similarly, if $g \in B(X)$, then $g$ is open, and if $g = \bigvee e^\flat[T]$, then there is a countable subset $T' \subseteq T$ with
$g = \bigvee e^\flat[T']$. Now, if $f,g \in B(X)$, with $f \le g$, then write $f = \bigwedge e^\flat[S]$ and $g  = \bigvee e^\flat[T]$
for some $S,T \subseteq C(Y)$. From the above, there are countable $S', T'$ with $f = \bigwedge e^\flat[S']$ and
$g = \bigvee e^\flat[T']$. Thus, $f \le \bigwedge e^\flat[S'] \le \bigvee e^\flat[T'] \le g$, which shows that $e^\flat$ satisfies (T).
\end{example}

We conclude this section by relating (N) to a condition formulated by  Blatter and Seever based on a technique of  Dieudonn\'{e} (see \cite[Lem.~1.2]{BS75} and \cite[p.~21]{Edw66}). What we call (D) is a coarser version of a condition that appears in \cite[Lem.~1.2]{BS75}.

\begin{definition}
Let $\alpha : A \to B$ be a basic extension, $f\in B$ closed, $g\in B$ open, and $0 < \varepsilon \in \mathbb{R}$.
\begin{enumerate}
\item[(D)]
If $f+\varepsilon \le g$, then there is $a \in \widehat{A}$ with $f \le \widehat{\alpha}(a) \le g$.
\end{enumerate}
\end{definition}

\begin{theorem} \label{thm: N = D}
A basic extension $\alpha : A \to B$ satisfies \emph{(N)} iff it satisfies \emph{(D)}.
\end{theorem}

\begin{proof}
That (N) implies (D) is clear. To see that (D) implies (N), we translate the argument of \cite[Lem.~1.2]{BS75} to our context. Let
$f$ be closed, $g$ open, and $f \le g$. We construct a sequence $a_n \in \widehat{A}$ such that for each $n$,
\begin{align}
f - 1/2^n &\le \widehat{\alpha}(a_n) \le g \\
a_n - 1/2^n &\le a_{n+1} \le a_n + 1/2^n
\end{align}
By (D), there is $a_1 \in \widehat{A}$ with $f - 1/2 \le \widehat{\alpha}(a_1) \le g$. Suppose that we have
$a_1, \dots, a_m \in \widehat{A}$ satisfying (1) for all $n \le m$ and (2) for all $n < m$. Therefore,
by (1) for $m$ we get
\[
f \le \widehat{\alpha}(a_m)  + 1/2^m
\]
Thus,
\[
f \vee (\widehat{\alpha}(a_{m}) - 1/2^{m+1}) \le \widehat{\alpha}(a_m)  + 1/2^m.
\]
Since $f\le g$, it is also clear that
\[
f \vee (\widehat{\alpha}(a_{m}) - 1/2^{m+1}) \le g.
\]
So
\[
f \vee (\widehat{\alpha}(a_{m}) - 1/2^{m+1}) \le g \wedge (\widehat{\alpha}(a_m)  + 1/2^m).
\]
But
\begin{equation}
f \vee (\widehat{\alpha}(a_{m}) -1/2^{m+1}) = (f - 1/2^{m+1}) \vee (\widehat{\alpha}(a_{m}) - 1/2^m) + 1/2^{m+1}.
\end{equation}
Consequently,
\begin{equation}
(f - 1/2^{m+1}) \vee (\widehat{\alpha}(a_{m}) - 1/2^m) + 1/2^{m-1} \le g \wedge (\widehat{\alpha}(a_m) + 1/2^m).
\end{equation}
By Remark~\ref{rem: closed}, $(f - 1/2^{m+1}) \vee (\widehat{\alpha}(a_{m}) - 1/2^m)$ is closed and $g \wedge (\widehat{\alpha}(a_m) + 1/2^m)$ is open. By (D), there is $a_{m+1} \in \widehat{A}$ satisfying
\[
(f - 1/2^{m+1}) \vee (\widehat{\alpha}(a_{m}) - 1/2^m) \le \widehat{\alpha}(a_{m+1}) \le g \wedge (\widehat{\alpha}(a_m) + 1/2^m).
\]
Therefore, (1) and (2) hold for $n = m+1$.
By induction we have produced the desired sequence. Condition (2) implies that $\{a_n\}$ is a Cauchy sequence, so has a uniform limit
$a \in \widehat{A}$. It follows from (1) that $f \le \widehat{\alpha}(a) \le g$. Thus, (N) holds.
\end{proof}

\begin{remark} \label{KT remark} The literature on insertion theorems is extensive and includes early theorems by Baire for the real line, Hahn for metric spaces, and Dieudonn\'{e} for paracompact spaces. These early theorems were generalized to the setting of normal spaces by Kat\v{e}tov \cite{Kat51, Kat53} and Tong \cite{Ton52}, resulting in what is now known as the Kat\v{e}tov-Tong Theorem. Our  approach mostly relies on the work of Tong \cite{Ton52} and later work of Blatter and Seever \cite{BS75}. A few other  references also have aspects that lend themselves to our algebraic approach. For example,  Kubiak \cite{Kub93} proves several results in a similar spirit of Tong's key lemma  (see Theorem~\ref{Tong}), and along the lines of our Condition (D), while Lane \cite{Lan76} uses algebraic arguments to move from strict insertion to insertion. There has also been a good deal of recent work on pointfree versions of insertions theorems; see for example \cite{Pic06}. While our approach is different than the pointfree one, it shares a similar goal of reformulating topological properties in algebraic settings. 
\end{remark}

\section{Compact and Lindel\"of extensions}

In this section we discuss compact basic extensions, which were first studied in \cite{BMO18c} under the name of canonical extensions. 
We show that compact extensions dually correspond to compact Hausdorff spaces, and utilize the compactness axiom to give an alternate
proof of a version of the Stone-Weierstrass theorem. We conclude the section by
introducing Lindel\"of basic extensions, and proving that they dually correspond to Lindel\"of spaces.
We start by recalling the compactness axiom from \cite[Def.~1.6(2)]{BMO18c}.

\begin{definition}
Let $\alpha : A \to B$ be a basic extension, $S,T \subseteq A$, and $\varepsilon > 0$.
\begin{enumerate}
\item[(C)] If $\bigwedge \alpha[S] + \varepsilon \le \bigvee \alpha[T]$, then there are finite subsets $S' \subseteq S$ and
$T'\subseteq T$ with $\bigwedge S' \le \bigvee T'$.
\end{enumerate}
\end{definition}

 \begin{remark} \label{rem: (C)}
Let $\alpha : A \to B$ be a basic extension.
\begin{enumerate}
\item If $\alpha$ satisfies (C), then since $\alpha$ preserves finite joins and meets and $S',T'$ are finite subsets of $A$,
the inequality $\bigwedge S' \le \bigvee T'$ in $A$ is equivalent to the inequality $\bigwedge \alpha[S'] \le \bigvee \alpha[T']$ in $B$.
\item The presence of $\varepsilon$ in (C) is necessary. For,
if $S = \{r \mid 0 < r \in \mathbb{R} \}$ and $T = \{0\}$, then $\bigwedge \alpha[S] \le \bigvee \alpha[T]$, but there is no finite subset
$S'$ of $S$ with $\bigwedge \alpha[S'] \le \bigvee \alpha[T]$. Therefore, no basic extension satisfies the condition:
If $\bigwedge \alpha[S] \le \bigvee \alpha[T]$, then there are finite subsets $S' \subseteq S$ and $T'\subseteq T$ with $\bigwedge S' \le \bigvee T'$.
\item By \cite[Lem.~2.4]{BMO18c}, (C) is equivalent to the following condition: If $T\subseteq A$, $\varepsilon>0$, and
$\varepsilon \le \bigvee \alpha[T]$,
then there is a finite subset $T' \subseteq T$ with $0 \le \bigvee \alpha[T']$.
\end{enumerate}
\end{remark}

\begin{definition}
\begin{enumerate}
\item[]
\item We call a basic extension $\alpha : A \to B$ \emph{compact} if $A \in \ubal$ and $\alpha$ satisfies (C).
\item  Let $\cbasic$ be the full subcategory of $\ubasic$ consisting of compact extensions.
\end{enumerate}
\end{definition}

It is clear that Condition (C) implies (D). Therefore, by Theorem~\ref{thm: N = D}, (C) implies (N). Thus, by Theorem~\ref{U2 implies normal},
$\cbasic$ is a full subcategory of $\nbasic$. The next theorem follows from \cite[Thm.~2.6(2)]{BMO18c}.

\begin{theorem} \label{thm: compact}
Let $e : X \to Y$ be a compactification. Then the following are equivalent.
\begin{enumerate}
\item $e^\flat : C(Y) \to B(X)$ is a compact extension.
\item $X$ is compact.
\item $e : X \to Y$ is a homeomorphism.
\end{enumerate}
\end{theorem}

Theorems~\ref{CR duality} and~\ref{thm: compact} yield the following duality theorem for compact Hausdorff spaces.

\begin{theorem} \label{thm: cbasic = KHaus}
$\cbasic$ is dually equivalent to $\KHaus$.
\end{theorem}

By Gelfand-Naimark-Stone duality, $\KHaus$ is dually equivalent to $\ubal$. Therefore, by Theorem~\ref{thm: cbasic = KHaus}, $\cbasic$ is
equivalent to $\ubal$. We give a direct proof of this result, thus obtaining
a different view of Gelfand-Naimark-Stone duality.

\begin{theorem}
$\cbasic$ is equivalent to $\ubal$.
\end{theorem}

\begin{proof}
We define a functor ${\sf F} : \cbasic \to \ubal$ by sending a compact extension $\alpha : A \to B$ to $A$, and a morphism $(\rho, \sigma)$
in $\cbasic$ to $\rho$. It is clear that $\sf F$ is a functor. It follows from \cite[Thm.~1.8(2)]{BMO18c} that
$\zeta_A : A \to B(Y_A)$ is a compact extension. Since ${\sf F}(\zeta_A) = A$, each $A \in \bal$ is in the image of $\sf F$.
To see that $\sf F$ is faithful, it is sufficient to show that if $(\rho, \sigma_1)$ and $(\rho, \sigma_2)$ are morphisms from
$\alpha : A \to B$ to $\alpha' : A' \to B'$, then $\sigma_1 = \sigma_2$.
\[
\begin{tikzcd}[column sep=5pc]
A \arrow[r, "\alpha"] \arrow[d, "\rho"'] & B \arrow[d, shift right = .30pc, "\sigma_1"'] \arrow[d, shift left = .30pc, "\sigma_2"] \\
A' \arrow[r, "\alpha'"'] & B'
\end{tikzcd}
\]
Both $\sigma_1$ and $\sigma_2$ are normal homomorphisms. Therefore, as $\alpha[A]$ is join-meet dense in $B$, the equation
$\sigma_1 \circ \alpha = \sigma_2 \circ \alpha$ yields $\sigma_1 = \sigma_2$. Finally, to see that $\sf F$ is full, if $\rho : A \to A'$
is a morphism in $\ubal$, then $(\rho, (\rho_*)^+)$ is a morphism from $\zeta_ A : A \to B(Y_A)$ to $\zeta_{A'} : A' \to B(Y_{A'})$
with ${\sf F}((\rho, (\rho_*)^+) = \rho$.
\[
\begin{tikzcd} [column sep=5pc]
A \arrow[r, "\zeta_A"] \arrow[d, "\rho"'] & B(Y_A) \arrow[d, "(\rho_*)^+"] \\
A' \arrow[r, "\zeta_{A'}"'] & B(Y_{A'})
\end{tikzcd}
\]
This shows that $\sf F$ is full, which completes the proof that $\sf F$ is an equivalence of categories.
\end{proof}

By the celebrated Stone-Weierstrass theorem, if $X$ is compact Hausdorff and $A$ is an $\mathbb{R}$-subalgebra of $C(X)$ which separates points
of $X$, then $A$ is uniformly dense in $C(X)$. A weaker version, restricting to $\ell$-subalgebras of $C(X)$, plays a central role in proving
Gelfand-Naimark-Stone duality. We show how to derive 
this version of the Stone-Weierstrass theorem by utilizing compact extensions. For this we need the following lemma.

\begin{lemma} \label{clopen}
Suppose that $\alpha : A \to B$ is a basic extension satisfying \emph{(C)}. Then the clopen elements of $B$ are in the uniform closure of
$\alpha[A]$ in $B$.
\end{lemma}

\begin{proof}
Let $f$ be clopen in $B$ and let $\varepsilon  >0$. Then $f+\varepsilon/2$ and $f+\varepsilon$ are also clopen. So $f+\varepsilon/2$ is closed,
$f+\varepsilon$ is open, and $(f + \varepsilon/2) + \varepsilon/2 \le f + \varepsilon$.
By (C), there is $a \in A$ such that $f \leq \alpha(a) \leq f +\varepsilon $. Therefore, $\| f - \alpha(a)\| \le \varepsilon$. This shows that
$f$ is in the uniform closure of $\alpha[A]$.
\end{proof}

\begin{theorem}[Stone-Weierstrass for $\ell$-subalgebras]
Let $X \in \KHaus$, $A \in \bal$, and $\alpha : A \to C(X)$ be a monomorphism in $\bal$ such that $\alpha[A]$ separates points of $X$.
Then $\alpha[A]$ is dense in $C(X)$.
\end{theorem}

\begin{proof}
It follows from \cite[Lem.~2.8(3)]{BMO18c} that $\alpha[A]$ is join-meet dense in $B(X)$, so $\alpha : A \to B(X)$ is a basic extension.
The basic extension $\iota : C(X) \to B(X)$ satisfies (C) by Theorem~\ref{thm: compact}. It is then immediate that $\alpha$ satisfies (C).
Let $\overline{A}$ be the uniform closure of $\alpha[A]$ in $B(X)$.
By \cite[Lem.~2.8(1 and 2)]{BMO18c}, every $f\in C(X)$ is clopen relative to $\alpha$. Therefore,
by Lemma~\ref{clopen}, $C(X) \subseteq \overline{A}$. Thus,
$\alpha[A]$ is uniformly dense in $C(X)$.
\end{proof}

Two important weakenings of the compactness condition are the Lindel\"{o}f and local compactness conditions. We conclude this section by introducing basic extensions corresponding to Lindel\"{o}f spaces. In the next section we introduce basic extensions corresponding to locally compact spaces.

We recall that a completely regular space $X$ is \emph{Lindel\"{o}f} provided each open cover of $X$ has a countable subcover. With this in mind, we introduce the following natural weakening of (C).

\begin{definition}
Let $\alpha : A \to B$ be a basic extension, $S,T\subseteq A$, and $\varepsilon > 0$.
\begin{enumerate}
\item[(L)] If $\bigwedge\alpha[S] + \varepsilon \le \bigvee\alpha[T]$, then there are countable subsets $S' \subseteq S$ and
$T'\subseteq T$ with $\bigwedge \alpha[S'] \le \bigvee \alpha[T']$.
\end{enumerate}
\end{definition}

\begin{theorem} \label{thm: removing epsilon}
For a basic extension $\alpha : A \to B$ the following are equivalent.
\begin{enumerate}
\item $\alpha$ satisfies \emph{(L)}.
\item If $T$ is a subset of $A$ with $\varepsilon \le \bigvee\alpha[T]$ for some $\varepsilon > 0$, then there is a countable subset
$T'\subseteq T$ with $0 \le \bigvee \alpha[T']$.
\item If $T$ is a subset of $A$ with $0 \le \bigvee \alpha[T]$, then there is a countable subset $T'\subseteq T$ with $0 \le \bigvee \alpha[T']$.
\item If $S,T$ are subsets of $A$ with $\bigwedge\alpha[S] \le \bigvee\alpha[T]$, then there are countable subsets $S' \subseteq S$ and
$T'\subseteq T$ with $\bigwedge \alpha[S'] \le \bigvee \alpha[T']$.
\end{enumerate}
\end{theorem}

\begin{proof}
(1) $\Rightarrow$ (2). This follows by setting $S = \{0\}$.

(2) $\Rightarrow$ (3). Let $T$ be a subset of $A$ with $0 \le \bigvee\alpha[T]$. Set $T+1/n = \{ c + 1/n \mid c \in T\}$. Then
$1/n  \le \bigvee \alpha[T+1/n]$. By (2), there are countable subsets $T_n \subseteq T$ such that $0 \le \bigvee \alpha[T_n+1/n]$.
Let $T' = \bigcup_n T_n$. Clearly $T'$ is countable, and
\begin{align*}
0 &\le \bigwedge_n (\bigvee \alpha[T_n+1/n]) = \bigwedge_n (\bigvee\alpha[T_n] + 1/n)\\
&\le \bigwedge_n [(\bigvee \alpha[T']) + 1/n] = \bigvee \alpha[T'] + \bigwedge_n 1/n  = \bigvee\alpha[T'].
\end{align*}
Therefore, (3) holds.

(3) $\Rightarrow$ (4). Let $S,T$ be subsets of $A$ with $\bigwedge \alpha[S]  \le \bigvee \alpha[T]$. Then, by Remark~\ref{rem: closed},
\begin{align*}
0 &\le \bigvee \alpha[T] - \bigwedge \alpha[S] = \bigvee \{ \alpha(t) - \alpha(s) \mid t \in T, s \in S\} \\
&= \bigvee \{ \alpha(t-s) \mid t \in T, s \in S\}.
\end{align*}
By (3), there is a countable collection $\{t_n - s_n\}$ with
\[
0 \le \bigvee_n \alpha(t_n - s_n) = \bigvee_n (\alpha(t_n) - \alpha(s_n)) \le \bigvee_n \left(\alpha(t_n) - \bigwedge_n\alpha(s_n)\right) = \bigvee_n\alpha(t_n) - \bigwedge_n \alpha(s_n).
\]
Thus, $\bigwedge \alpha(s_n) \le \bigvee \alpha(t_n)$, so (4) holds.

(4) $\Rightarrow$ (1). This is clear.
\end{proof}

\begin{theorem} \label{thm: L = Lind}
Let $e : X \to Y$ be a compactification and $e^\flat : C(Y) \to B(X)$ the corresponding basic extension. Then $X$ is Lindel\"{o}f iff
$e^\flat$ satisfies \emph{(L)}.
\end{theorem}

\begin{proof}
Suppose $e^\flat$ satisfies (L). Let $\mathcal{U}$ be an open cover of $X$. If $U \in \mathcal{U}$, then $\chi_U$ is open with respect to
$e^\flat$ by Lemma~\ref{lem: closed}(2). Therefore, $\chi_U - 1$ is open, and so there is $S_U \subseteq C(Y)$ with
$\chi_U - 1 = \bigvee e^\flat[S_U]$. Since $\mathcal{U}$ is a cover of $X$, $\bigvee \{ \chi_U \mid U \in \mathcal{U} \} = 1$.
Set $S = \bigcup \{ S_U \mid U \in \mathcal{U} \}$. Then
\[
\bigvee e^\flat[S] = \bigvee \{ \bigvee e^\flat[S_U] \mid U \in \mathcal{U} \}  = \bigvee \{ \chi_U - 1 \mid U \in \mathcal{U} \} = \bigvee \{\chi_U \mid U \in \mathcal{U} \} - 1  = 0. 
\]
Since (L) holds, by Theorem~\ref{thm: removing epsilon},
there is a countable subset $S'$ of $S$ with $0 \le \bigvee e^\flat[S']$. Let $s \in S'$. Then there is
$U \in \mathcal{U}$ with $e^\flat(s) \le \chi_U - 1$. Therefore, there is a countable subset $\mathcal{V}$ of $\mathcal{U}$ such that
\[
0 \le \bigvee e^\flat[S'] \le \bigvee \{ \chi_U -1 \mid U \in \mathcal{V} \} \le \bigvee \{ \chi_U - 1 \mid U \in \mathcal{U} \} = 0.
\]
Thus, $1 = \bigvee \{ \chi_U \mid U \in \mathcal{V} \}$, which shows that
$\mathcal{V}$ is a countable subcover of $X$. Consequently, $X$ is Lindel\"{o}f.

Conversely, suppose that $X$ is Lindel\"of. Let $0 < \varepsilon \le \bigvee e^\flat[T]$ for some subset $T$ of $C(Y)$. If $x \in X$, then there is $g \in T$ with $g(e(x)) > \varepsilon/2$, which yields $X = \bigcup \{e^{-1}g^{-1}(\varepsilon/2, \infty) \mid g \in T\}$. Since $X$ is Lindel\"of, there is a countable subset $\{g_n\}$ of $T$ with $X = \bigcup_n e^{-1}g_n^{-1}(\varepsilon/2, \infty)$. Therefore, for each $x \in X$ there is $n$ with $g_n(e(x)) > \varepsilon/2$. This means that $\bigvee_n e^\flat(g_n) > \varepsilon/2 > 0$. Thus, $e^\flat$ satisfies (L) by Theorem~\ref{thm: removing epsilon}.
\end{proof}

There exist basic extensions that satisfy (L), but are not maximal, as the following example shows.

\begin{example}
Let $e^\flat : C(Y) \to B(X)$ be the basic extension of
Example~\ref{ex: (T) weaker than (N)}. Then $X$ is Lindel\"{o}f since $X$ is countable. Moreover, the argument of the example shows that
any $f \in B(X)$ that is a meet (resp.~join) from $e^\flat[C(Y)]$ can be written as a countable meet (resp.~join) by picking an
appropriate subset. From this it follows that $e^\flat$ satisfies (L). It is not a maximal extension since $e$ is not isomorphic
to the Stone-\v{C}ech compactification $s : X \to \beta X$.
\end{example}

It follows that (L) does not imply maximality, and so neither does it imply normality. Since every Lindel\"{o}f space is normal
(see, e.g., \cite[Thm.~3.8.2]{Eng89}), this indicates that we need a stronger condition to have an appropriate notion of a Lindel\"{o}f
basic extension.

\begin{definition}
Let $\alpha : A \to B$ be a basic extension and $S,T\subseteq A$.
\begin{enumerate}
\item[(SL)] If $\bigwedge \alpha[S] \le \bigvee \alpha[T]$, then there are countable subsets $S' \subseteq S$, $T' \subseteq T$, and
$a \in A$ with $\bigwedge \alpha[S'] \le \alpha(a) \le \bigvee \alpha[T']$.
\end{enumerate}
\end{definition}

\begin{theorem} \label{thm: Lind}
Let $e : X \to Y$ be a compactification and $e^\flat : C(Y) \to B(X)$ the corresponding basic extension. Then $e^\flat$ satisfies \emph{(SL)}
iff $X$ is Lindel\"{o}f and $e$ is isomorphic to $s : X \to \beta X$.
\end{theorem}

\begin{proof}
It is straightforward to see
that $e^\flat$ satisfies (SL) iff it satisfies both (L) and (N). The result then follows from Theorems~\ref{thm: L = Lind}
and \ref{U2 implies normal}.
\end{proof}

\begin{definition}
\begin{enumerate}
\item[]
\item We call a basic extension $\alpha : A \to B$ \emph{Lindel\"{o}f} provided $A \in \ubal$ and $\alpha$ satisfies (SL).
\item Let $\lbasic$ be the full subcategory of $\ubasic$ consisting of Lindel\"{o}f extensions.
\end{enumerate}
\end{definition}

\begin{remark}
If $\alpha : A \to B$ is a Lindel\"{o}f extension, then $\alpha$ satisfies (N). Therefore,
$\lbasic$ is a full subcategory of $\nbasic$.
\end{remark}

\begin{definition}
Let $\Lind$ be the full subcategory of $\creg$ consisting of Lindel\"{o}f spaces.
\end{definition}

The following theorem follows immediately from 
Theorems~\ref{thm: norm} and~\ref{thm: Lind}.

\begin{theorem}
$\Lind$ is dually equivalent to $\lbasic$.
\end{theorem}

\section{locally compact extensions}

In this final section we define locally compact basic extensions and prove that they dually correspond to locally compact spaces.
We then conclude by characterizing the one-point compactification of a locally compact Hausdorff space by means of minimal basic extensions.

\begin{definition} \label{def: alpha compact}
Let $\alpha : A \to B$ be a basic extension.
\begin{enumerate}
\item We call $b \in B$ \emph{$\alpha$-compact} if for $\varepsilon > 0$ and a subset $T$ of nonnegative elements of $A$, from
$|b| + \varepsilon \le \bigvee \alpha[T]$ it follows that there is a finite subset $T'$ of $T$ with $|b| \le \bigvee \alpha[T']$.
\item We call $a \in A$ \emph{$\alpha$-compact} if $\alpha(a)$ is $\alpha$-compact.
\end{enumerate}
\end{definition}

\begin{remark}
By Remark~\ref{rem: (C)}(3), a basic extension $\alpha : A \to B$ satisfies (C) if for $\varepsilon > 0$ and $T\subseteq A$,
from $\varepsilon \le \bigvee \alpha[T]$ it follows that there is a finite $T' \subseteq T$ with $0 \le \bigvee\alpha[T']$.
Thus, $0$ being $\alpha$-compact is weaker than $\alpha$ satisfying (C) since in (C) the set $T \subseteq A$ need not consist
of nonnegative elements.
\end{remark}

Definition~\ref{def: alpha compact} is motivated by the following lemma.

\begin{lemma} \label{lem: compact elements}
Let $e : X \to Y$ be a compatification and $\alpha = e^\flat : C(Y) \to B(X)$ the corresponding basic extension. A subset $F$ of $X$ is
compact iff $\chi_F$ is $\alpha$-compact.
\end{lemma}

\begin{proof}
Let $F\subseteq X$. First suppose that $\chi_F$ is $\alpha$-compact. If $\mathcal{U}$ is an open cover of $F$, then we may write
$\mathcal{U} = \{ e^{-1}(V) \mid V \in \mathcal{V} \}$ for some family $\mathcal V$ of open sets of $Y$. We have
\[
\chi_F \le \bigvee \{ \chi_U \mid U \in \mathcal{U} \} = \bigvee \{ \alpha(\chi_V) \mid V \in \mathcal{V} \},
\]
so
\[
(\chi_F + 1/2)  \le \bigvee \{ \chi_U + 1/2 \mid U \in \mathcal{U} \} = \bigvee \{ \alpha(\chi_V) + 1/2 \mid V \in \mathcal{V} \}.
\]
Since $\chi_F$ is $\alpha$-compact, there are $V_1, \dots, V_n \in \mathcal{V}$ with
\[
\chi_F  \le (\alpha(\chi_{V_1}) + 1/2) \vee \cdots \vee (\alpha(\chi_{V_n}) + 1/2) =
(\alpha(\chi_{V_1}) \vee \cdots \vee \alpha(\chi_{V_n})) + 1/2 = \chi_U + 1/2,
\]
where $U = e^{-1}(V_1) \cup \cdots \cup e^{-1}(V_n)$. Thus, $F \subseteq U$, and so $F$ is compact.

Conversely, suppose that $F$ is compact and $\chi_F + \varepsilon \le \bigvee \alpha[T]$ for $\varepsilon > 0$ and a set $T$ of nonnegative
elements of $C(Y)$. Then, for each $x \in F$, there is $g \in T$ with $g(e(x)) > 1 + \varepsilon/2$. Therefore,
$\{ e^{-1}g^{-1}(1 + \varepsilon/2, \infty)\}$ is an open cover of $F$. Since $F$ is compact, there is a finite subcover, say
$F \subseteq \bigcup_{i=1}^n e^{-1}g_i^{-1}(1+\varepsilon/2, \infty)$. Thus, $(\alpha(g_1) \vee \cdots \vee \alpha(g_n))(x) \ge \chi_F(x)$
for each $x \in X$. This inequality also holds for all $x \notin F$ since $\chi_F(x) = 0$ and the $g_i$ are nonnegative. Consequently,
$\chi_F \le \alpha(g_1) \vee \cdots \vee \alpha(g_n)$, and so $\chi_F$ is $\alpha$-compact.
\end{proof}

Further connection between Condition (C) and $\alpha$-compactness is given in Proposition~\ref{(C) versus alpha compact}, which requires
the following lemma.

\begin{lemma} \label{lem: auxiliary}
Let $A\in\bal$ and $a,b\in A$. If $a \le b \vee (a-1)$, then $a \le b$.
\end{lemma}

\begin{proof}
Suppose that $a \le b \vee (a-1)$. Multiplying by $-1$ gives $(1-a) \wedge -b \le -a$, and adding $a$ yields $1 \wedge (a-b) \le 0$.
Therefore, $1 \wedge [(a-b)\vee 0] = 0$. Since $1$ is a strong order-unit, by \cite[p.~308, Lem.~4]{Bir79}, $(a-b) \vee 0 = 0$, which gives
$a-b \le 0$. Thus, $a \le b$.
\end{proof}

\begin{proposition} \label{(C) versus alpha compact}
For a basic extension $\alpha : A \to B$, the following are equivalent.
\begin{enumerate}
\item $\alpha$ satisfies \emph{(C)}.
\item $r$ is $\alpha$-compact for every $r \in \mathbb{R}$.
\item $1$ is $\alpha$-compact.
\end{enumerate}
\end{proposition}

\begin{proof}
(1) $\Rightarrow$ (2). Suppose $r\in\mathbb{R}$ and $|r| + \varepsilon \le \bigvee \alpha[T]$ for $\varepsilon > 0$ and a subset $T$ of
nonnegative elements of $A$. Set $S = \{|r|\}$. By (C), there is a finite $T' \subseteq T$ with $|r| \le \bigvee \alpha[T']$. Thus, $r$
is $\alpha$-compact.

(2) $\Rightarrow$ (3). This is clear.

(3) $\Rightarrow$ (1). Suppose that $\varepsilon \le \bigvee \alpha[T]$ for $\varepsilon > 0$ and $T \subseteq A$. Set
$S = \{ (t + 1) \vee 0 \mid t \in T\}$. Then $1 + \varepsilon \le 1 + \bigvee \alpha[T] \le \bigvee \alpha[S]$. Since $S$ consists of nonnegative
elements of $A$ and 1 is $\alpha$-compact, there is a finite $S' \subseteq S$ with $1 \le \bigvee \alpha[S']$. Let
$T' \subseteq T$ be such that $S' = \{ (t+1) \vee 0 \mid t \in T'\}$. Then $1 \le \bigvee \alpha[S'] = (\bigvee \alpha[T'] + 1) \vee 0$.
By Lemma~\ref{lem: auxiliary}, $1 \le \bigvee (\alpha[T'] + 1) = \bigvee \alpha[T'] + 1$. Therefore, $0 \le \bigvee \alpha[T']$. Thus,
$\alpha$ satisfies (C) by Remark~\ref{rem: (C)}(3).
\end{proof}

\begin{remark} \label{rem: alpha compact}
For a basic extension $\alpha : A \to B$, the set of $\alpha$-compact elements of $B$ does not form an $\ell$-ideal in general. To see this,
let $\alpha = e^\flat : C(Y) \to B(X)$ be the basic extension corresponding to a compactification $e : X \to Y$. By
Lemma~\ref{lem: compact elements}, if $F$ is a compact subset of $X$, then $\chi_F$ is $\alpha$-compact. However, if $S$ is a nonempty subset
of $F$ which is not compact, then $\chi_S$ is not $\alpha$-compact, while $0 \le \chi_S \le \chi_C$.
Thus, in general, the set of $\alpha$-compact elements of $B$ is not an $\ell$-ideal of $B$.
\end{remark}

While Remark~\ref{rem: alpha compact} shows that the set of $\alpha$-compact elements of $B$ does not in general form an $\ell$-ideal of $B$, we show that the set of $\alpha$-compact elements of $A$ forms an $\ell$-ideal of $A$. 

\begin{definition}
For a basic extension $\alpha : A \to B$, let $I_\alpha =  \{ a \in A \mid a \textrm{ is } \alpha\textrm{-compact} \}$.
\end{definition}

\begin{lemma} \label{lem: Ialpha}
Let $\alpha : A \to B$ be a basic extension. Then $I_\alpha$ is an $\ell$-ideal of $A$.
\end{lemma}

\begin{proof}
It is clear that $0 \in I_\alpha$. Suppose that $b \in I_\alpha$ and $|a| \le |b|$. If $\alpha(|a|) + \varepsilon \le \bigvee \alpha[T]$ for
$\varepsilon > 0$ and a set $T$ of nonnegative elements of $A$, then $\alpha(|b|) + \varepsilon \le \bigvee \{ \alpha(t + |b|-|a|) \mid t \in T\}$.
Since $b \in I_\alpha$ there are $t_1,\dots, t_n \in T$ with $\alpha(|b|) \le \alpha(t_1+|b|-|a|) \vee \cdots \vee \alpha(t_n + |b|-|a|)$. Therefore,
$\alpha(|b|) \le \alpha(t_1)\vee \cdots \vee \alpha(t_n) + \alpha(|b|) - \alpha(|a|)$, so $\alpha(|a|) \le \alpha(t_1)\vee \cdots \vee \alpha(t_n)$. Thus,
$a \in I_\alpha$.

To show that $I_\alpha$ is closed under $+$, we show that it is closed under joins and doubling. Using the $\ell$-group inequality
$|a+b| \le 2(|a| \vee |b|)$ will then show that $I_\alpha$ is closed under $+$ by the previous paragraph. Suppose that $a,b \in I_\alpha$.
If $\alpha(|a \vee b|) + \varepsilon \le \bigvee \alpha[T]$ for $\varepsilon > 0$ and a set $T$ of nonnegative elements of $A$, then as
$|a| + \varepsilon, |b| + \varepsilon \le |a\vee b| + \varepsilon$, there are finite subsets $T_1, T_2$ of $T$ with
$\alpha(|a|) \le \bigvee \alpha[T_1]$ and $\alpha(|b|) \le \bigvee \alpha[T_2]$. Therefore,
$\alpha(|a \vee b|) \le \alpha(|a|) \vee \alpha(|b|) \le \bigvee \alpha[T_1 \cup T_2]$. Thus, $a\vee b \in I_\alpha$. Next, let $a \in I_\alpha$.
If $\alpha(2|a|) + \varepsilon \le \bigvee \alpha[T]$ for $\varepsilon > 0$ and a set $T$ of nonnegative elements of $A$, then
$\alpha(|a|) + \varepsilon/2 \le \bigvee \{ \alpha(t/2) \mid t \in T\}$. Therefore, there is a finite subset $T'$ of $T$ with
$\alpha(|a|) \le \alpha(t_1/2) \vee \cdots \vee \alpha(t_n/2)$. Thus, $\alpha(|a|) \le (\alpha(t_1) \vee \cdots \vee \alpha(t_n))/2$, and so
$\alpha(|2a|) \le \alpha(t_1) \vee \cdots \vee \alpha(t_n)$. Consequently, $2a \in I_\alpha$, and so $I_\alpha$ is closed under addition. Finally, suppose that $a \in A$ and $b \in I_\alpha$. Since $A$ is bounded, there is $n \in \mathbb{N}$ with $|a| \le n$. Therefore, $|ab| \le |a||b| \le n|b| = |nb|$, so $ab \in I_\alpha$. Thus, $I_\alpha$ is an $\ell$-ideal of $A$.
\end{proof}

\begin{proposition}
For a basic extension $\alpha : A \to B$, the three conditions of Proposition~\ref{(C) versus alpha compact} are equivalent to:
\begin{enumerate}
\item[(4)] $I_\alpha = A$. 
\end{enumerate}
\end{proposition}
\begin{proof}
If
1 is $\alpha$-compact,
since $I_\alpha$ is an $\ell$-ideal of $A$ by Lemma~\ref{lem: Ialpha},
we see that $I_\alpha = A$. Conversely, if $I_\alpha = A$, then $1 \in I_\alpha$, so 1 is $\alpha$-compact.
\end{proof}

Let $e : X \to Y$ be a compactification and $\alpha = e^\flat : C(Y) \to B(X)$ the corresponding basic extension. We next determine when $f \in C(Y)$ belongs to $I_\alpha$. For this we recall (see, e.g.,\cite[Sec.~1.10]{GJ60}) that
\[
Z(f) = \{ y \in Y \mid f(y) = 0\}.
\]
We set
\[
\coz(f) = Y \setminus Z(f),
\]
and for an ideal $I$ of $C(Y)$,
\[
Z(I) = \bigcap \{ Z(f) \mid f \in I\}.
\]

\begin{remark} \label{rem: Z versus Zl}
Since $\xi_Y : Y \to Y_{C(Y)}$ is a homeomorphism, for each $\ell$-ideal $I$ of $C(Y)$, we have $Z_\ell(I) = \xi_Y[Z(I)] $.
\end{remark}

\begin{lemma} \label{lem: characterization of Ialpha}
Let $e : X \to Y$ be a compactification and $\alpha = e^\flat : C(Y) \to B(X)$ the corresponding basic extension. For $f \in C(Y)$,
the following are equivalent.
\begin{enumerate}
\item $f \in I_\alpha$.
\item $\cl_Y(\coz(f)) \subseteq e[X]$.
\item $e^{-1}\cl_Y(\coz(f))$ is a compact subset of $X$.
\end{enumerate}
\end{lemma}

\begin{proof}
(1) $\Rightarrow$ (2). Suppose $f \in I_\alpha$. To show $\cl_Y(\coz(f)) \subseteq e[X]$, it suffices to assume $f \ge 0$ since
$\coz(f) = \coz(|f|)$. Suppose there is $y \in \cl_Y(\coz(f)) \setminus e[X]$. For each $z \in Y$ with $z \ne y$ there is an open
set $V_z$ of $Y$ with $y \in V_z$ and $z \notin \cl_Y(V_z)$. Therefore, there is $0 \le g_z \in C(Y)$ with $g_z = 0$ on $\cl_Y(V_z)$
and $g_z(z) = f(z) + 1$. It then follows that $\alpha(f) + 1 \le \bigvee \{ \alpha(g_z) \mid z \ne y\}$. Since $f \in I_\alpha$, there
are $z_1,\dots, z_n \in Y$ with $\alpha(f) \le \alpha(g_{z_1}) \vee \cdots \vee \alpha(g_{z_n}) = \alpha(g_{z_1} \vee \cdots \vee g_{z_n})$.
Now, $g_{z_1} \vee \cdots \vee g_{z_n}$ is zero on a neighborhood $V$ of $y$. Since $y \in \cl_Y(\coz(f))$, we have
$V \cap \coz(f) \ne \varnothing$. Therefore, since $e[X]$ is dense in $Y$, $V \cap \coz(f) \cap e[X] \ne \varnothing$. If $x \in X$
with $e(x) \in V \cap \coz(f)$, then $f(e(x)) > 0$ while $(g_{z_1} \vee \cdots \vee g_{z_n})(e(x)) = 0$. The obtained contradiction
proves that $\cl_Y(\coz(f)) \subseteq e[X]$.

(2) $\Rightarrow$ (3). Since $\cl_Y(\coz(f))$ is a closed subset of $Y$, it is compact. If it is contained in $e[X]$, then
$e^{-1}\cl_Y(\coz(f))$ is a compact subset of $X$ since $e$ is a homeomorphism from $X$ to $e[X]$.

(3) $\Rightarrow$ (1). It is sufficient to show that $|f| \in I_\alpha$, hence we assume that $f \ge 0$. Suppose that $F := e^{-1}\cl_Y(\coz(f))$ is a compact subset of $X$.
Let $\varepsilon > 0$ and $T$ be a subset of nonnegative elements of $C(Y)$ with $\alpha(f) + \varepsilon \le \bigvee \alpha[T]$. Then
\[
X = \bigcup \{ \alpha(g-f)^{-1}(\varepsilon/2, \infty) \mid g \in T\},
\]
so $F$ is covered by these open sets. Since $F$ is a compact subset of $X$, there are $g_1,\dots, g_n \in T$ with
$F \subseteq \alpha(g_1-f)^{-1}(\varepsilon/2, \infty) \cup \cdots \cup \alpha(g_n-f)^{-1}(\varepsilon/2, \infty)$.
As $g_1 \vee \cdots \vee g_n$ is a nonnegative function, if $x \notin F$, then $f(e(x)) = 0$, so
$\alpha(f)(x) \le \alpha(g_1\vee \cdots \vee g_n)(x)$. If $x \in F$, then $x \in \alpha(g_i-f)^{-1}(\varepsilon/2, \infty)$ for some $i$,
so $\alpha(g_i - f)(x) > \varepsilon/2 > 0$. This yields $\alpha(f)(x) \le \alpha(g_i)(x) \le  (\alpha(g_1)\vee \cdots \vee \alpha(g_n))(x)$. Thus,
$\alpha(f) \le \alpha(g_1) \vee \cdots \vee \alpha(g_n)$, and hence $f \in I_\alpha$.
\end{proof}

\begin{lemma} \label{lem: remainder}
Let $e : X \to Y$ be a compactification and $\alpha = e^\flat : C(Y) \to B(X)$ the corresponding basic extension. Then
$Z(I_\alpha) = \cl_Y(Y \setminus e[X])$.
\end{lemma}

\begin{proof}
Let $y \in Y \setminus e[X]$. If $f \in I_\alpha$, then $\coz(f) \subseteq \cl_Y(\coz(f)) \subseteq e[X]$ by
Lemma~\ref{lem: characterization of Ialpha}. Therefore, $f(y) = 0$, and so $y\in Z(f)$. Since $f\in I_\alpha$ was arbitrary,
we have $y\in Z(I_\alpha)$. Because this is true for every $y\in Y \setminus e[X]$, we see that $Y \setminus e[X] \subseteq Z(I_\alpha)$.
As $Z(I_\alpha)$ is closed in $Y$, we conclude that $\cl_Y(Y \setminus e[X]) \subseteq Z(I_\alpha)$.

For the reverse inclusion, suppose that $y \notin \cl_Y(Y \setminus e[X])$. Then $y \in V := Y \setminus \cl_Y(Y \setminus e[X]) = \int_Y(e[X])$,
an open subset of $Y$ contained in $e[X]$. Since $Y$ is regular, there is an open set $U$ of $Y$ with $y \in U$ and $\cl_Y(U) \subseteq V$.
Set $F = Y \setminus U$, a closed subset of $Y$. Since $y \notin F$, there is $f \in C(Y)$ with $0 \le f \le 1$, $f(y) = 1$, and
$f(F) = \{0\}$. Therefore, $\coz(f) \subseteq Y \setminus F = U$, so $\cl_Y(\coz(f)) \subseteq \cl_Y(U) \subseteq e[X]$. Thus, $f \in I_\alpha$
by Lemma~\ref{lem: characterization of Ialpha}. Because $f(y) = 1$, we have $y\notin Z(f)$, so $y \notin Z(I_\alpha)$.
\end{proof}

In the next theorem we characterize locally compact Hausdorff spaces in terms of basic extensions.

\begin{theorem} \label{thm: local compactness}
Let $e : X \to Y$ be a compactification and $\alpha = e^\flat : C(Y) \to B(X)$ the corresponding basic extension. Then $X$ is locally compact
iff $\alpha(b) = \bigvee \{ \alpha(a) \mid a \in I_\alpha, a \le b \}$ for all $0 \le b \in C(Y)$.
\end{theorem}

\begin{proof}
Suppose that $X$ is locally compact and $0\le b \in C(Y)$. Let $x \in X$. We show there is $a \in I_\alpha$ with $a \le b$ and
$\alpha(b)(x) = \alpha(a)(x)$. This is clear if $\alpha(b)(x) = 0$, so assume $\alpha(b)(x) > 0$. Since $X$ is locally compact,
there is an open neighborhood $U$ of $x$ with $\cl_X(U)$ compact. Therefore, $\cl_Y(e[U]) \subseteq e[X]$. There is an open set
$V$ of $Y$ with $U = e^{-1}(V)$. Because $\cl_Y(V) = \cl_Y(e[U])$, we see that $\cl_Y(V) \subseteq e[X]$. Since
$e(x) \notin Y \setminus V$, there is $0 \le a \in C(Y)$ with $a = 0$ on $Y \setminus V$ and $a(e(x)) = b(e(x))$. By replacing
$a$ by $a\wedge b$, we may assume $a \le b$. By construction, $\coz(a) \subseteq V$, so $\cl_Y(\coz(a)) \subseteq \cl_Y(V) \subseteq e[X]$.
Therefore, $a \in I_\alpha$ by Lemma~\ref{lem: characterization of Ialpha}. Thus,
$\alpha(b) = \bigvee \{ \alpha(a) \mid a \in I_\alpha, a \le b \}$.

For the converse, to see that $X$ is locally compact, by \cite[Cor.~3.3.11]{Eng89}, it suffices to show that $Y \setminus e[X]$
is closed. By Lemma~\ref{lem: remainder}, it is enough to show that $Z(I_\alpha) \subseteq Y\setminus e[X]$. To see this, we have
$1 = \bigvee \{ \alpha(a) \mid a \in I_\alpha, a \le 1\}$. From this equality it follows that if $x \in X$, then there is
$a \in I_\alpha$ with $\alpha(a)(x) \ne 0$. Consequently, $e(x) \notin Z(a)$, and hence $e(x)\notin Z(I_\alpha)$.
\end{proof}

Theorem~\ref{thm: local compactness} motivates the following definition.

\begin{definition}
We call a basic extension $\alpha : A \to B$ \emph{locally compact} provided
\[
\alpha(b) = \bigvee \{ \alpha(a) \mid a \in I_\alpha, a \le b\} \textrm{ for all }0 \le b \in A.
\]
\end{definition}

\begin{remark}
The assumption that $0 \le b$ is essential. To see this, let $A=C(Y)$ and let $\alpha(-1) = \bigvee \{ \alpha(a) \mid a \in I_\alpha, a \le -1\}$.
Then each such $a$ is bounded away from 0, and hence is a unit, forcing $I_\alpha = A$.
\end{remark}

We next show that a basic extension $\alpha$ is locally compact iff so is $\widehat{\alpha}$. For this we require the following lemma.

\begin{lemma} \label{lem: alpha lc iff alpha hat lc}
Let $\alpha : A \to B$ be a basic extension. Then $\zeta_A^{-1}(I_{\widehat{\alpha}}) = I_\alpha$.
\end{lemma}

\begin{proof}
Let $b \in I_\alpha$. Suppose for $\varepsilon > 0$ and a set $T$ of nonnegative elements of $\widehat{A}$ we have
$|\widehat{\alpha}(\zeta_A(b))| + \varepsilon \le \bigvee \widehat{\alpha}[T]$. By the proof of Lemma~\ref{lem: alpha versus alpha hat}, for each $t \in T$
we may find an increasing sequence $\{a_n\}$ of nonnegative elements of $A$ such that $t = \bigvee_n \zeta_A(a_n)$. If $S$ is the set
of all these elements, as $t$ ranges over $T$, from $\widehat{\alpha} \circ \zeta_A = \alpha$, we have
$|\alpha(b)| + \varepsilon \le \bigvee \alpha[S]$. Since $b\in I_\alpha$, there is a finite $S' \subseteq S$ with
$|\alpha(b)| \le \bigvee \alpha[S']$. For each $s \in S'$ there is $t \in T$ with $\alpha(s) \le \widehat{\alpha}(t)$. Thus, there is a finite subset
$T' \subseteq T$ with $\bigvee\alpha[S'] \le \bigvee \widehat{\alpha}[T']$, and so
$|\widehat{\alpha}(\zeta_A(b))| \le \bigvee \widehat{\alpha}[T']$. This shows $\zeta_A(b) \in I_{\widehat{\alpha}}$.

Conversely, suppose that $b \in A$ with $\zeta_A(b) \in I_{\widehat{\alpha}}$. Let $\varepsilon > 0$ and $T$ be a set of nonnegative
elements of $A$ with $|\alpha(b)| + \varepsilon \le \bigvee \alpha[T]$. Then
$|\widehat{\alpha}(\zeta_A(b))| + \varepsilon \le \bigvee \widehat{\alpha}[\zeta_A[T]]$.
Since $\zeta_A(b) \in I_{\widehat{\alpha}}$, there is a finite subset $T'$ of $T$
with $|\widehat{\alpha}(\zeta_A(b))| \le \bigvee \widehat{\alpha}[\zeta_A[T']]$, and so $|\alpha(b)| \le \bigvee \alpha[T']$.
Thus, $b \in I_\alpha$.
\end{proof}

\begin{lemma} \label{lem: alpha lc iff alpha hat lc}
A basic extension $\alpha:A\to B$ is locally compact iff $\widehat{\alpha}:\widehat{A}\to B$ is locally compact.
\end{lemma}

\begin{proof}
Let $\alpha$ be locally compact and let $0 \le b \in \widehat{A}$.
By the proof of Lemma~\ref{lem: alpha versus alpha hat}, we may find an increasing sequence $\{a_n\}$ of nonnegative elements of $A$ such that
$\widehat{\alpha}(b) = \bigvee_n \alpha(a_n)$.
Since $\alpha$ is locally compact, $\alpha(a_n) = \bigvee \{ \alpha(a) \mid a \in I_\alpha, a \le a_n\}$. Therefore,
\begin{align*}
\widehat{\alpha}(b) &= \bigvee \{ \alpha(a) \mid a \in I_\alpha, a \le a_n \textrm{ for some } n \} \\
&= \bigvee \{ \widehat{\alpha}(\zeta_A(a)) \mid a \in I_\alpha, a \le a_n \textrm{ for some } n \}.
\end{align*}
By Lemma~\ref{lem: alpha lc iff alpha hat lc}, $\widehat{\alpha}(b) = \bigvee\{ \widehat{\alpha}(c) \mid c \in I_{\widehat{\alpha}}, c \le b \}$.
Thus, $\widehat{\alpha}$ is locally compact.

Conversely, suppose that $\widehat{\alpha}$ is locally compact. Let $0 \le b \in A$. Then
\[
\widehat{\alpha}(\zeta_A(b)) = \bigvee \{ \widehat{\alpha}(c) \mid c \in I_{\widehat{\alpha}}, c \le \zeta_A(b) \}.
\]
By Lemma~\ref{lem: alpha versus alpha hat}(2), for each $c \in I_{\widehat{\alpha}}$ with $c \le \zeta_A(b)$, we may write
$\widehat{\alpha}(c) = \bigvee_n \alpha(a_n)$ for some $0 \le a_n \in A$ with $\zeta_A(a_n) \le c$.
By Lemma~\ref{lem: alpha lc iff alpha hat lc}, each $a_n$ belongs to $I_\alpha$. If $\zeta_A(a) \le c \le \zeta_A(b)$,
then $a \le b$ since $\zeta_A$ is an order embedding. Therefore,
\[
\alpha(b) = \widehat{\alpha}(\zeta_A(b)) = \bigvee \{ \alpha(a) \mid a \in I_\alpha, a \le b \}.
\]
Thus, $\alpha$ is locally compact.
\end{proof}

As a consequence, we obtain the following theorem.

\begin{theorem}
Let $\alpha : A \to B$ be a basic extension. Then $X_B$ is locally compact iff $\alpha$ is locally compact.
\end{theorem}

\begin{proof}
By 
Theorems~\ref{CR duality} and~\ref{thm: local compactness}, $X_B$ is locally compact iff $\widehat{\alpha}$ is locally compact.
By Lemma~\ref{lem: alpha lc iff alpha hat lc}, $\alpha$ is locally compact iff $\widehat{\alpha}$ is locally compact. The result
follows.
\end{proof}

It is worth pointing out that not every locally compact basic extension is maximal, as the next example shows.

\begin{example}
Consider the compactification $e : X \to Y$ of Example~\ref{ex: (T) weaker than (N)}. Clearly $X$ is locally compact, hence by
Theorem~\ref{thm: local compactness}, $e^\flat : C(Y ) \to B(X)$ is locally compact. Since $e$ is not isomorphic to the
Stone-\v{C}ech compactification of $X$, the basic extension $e^\flat$ is not maximal.
\end{example}

\begin{definition}
\begin{enumerate}
\item[]
\item Let $\lcbasic$ be the full subcategory of $\mbasic$ consisting of locally compact basic extensions.
\item Let $\LKHaus$ be the full subcategory of $\creg$ consisting of locally compact spaces.
\end{enumerate}
\end{definition}

The following is an immediate consequence of 
Theorems~\ref{CR duality} and~\ref{thm: local compactness}.

\begin{theorem}
There is a dual equivalence between $\LKHaus$ and $\lcbasic$.
\end{theorem}

We conclude the paper by characterizing one-point compactifications of locally compact spaces by means of minimal extensions.
For this we recall from the preliminaries that two basic extensions $\alpha : A \to B$ and $\gamma : C \to B$
are compatible if the topologies $\tau_\alpha$ and $\tau_\gamma$ on $X_B$ are equal. It is easy to see that
a basic extension $\alpha : A \to B$ is compatible with $\widehat{\alpha} : \widehat{A} \to B$.

\begin{definition}
We say that a basic extension $\alpha : A \to B$ is \emph{minimal} provided for every compatible basic extension $\gamma : C \to B$ with
$C \in \ubal$, there is a morphism $\delta : A \to C$ in $\bal$ such that $\gamma \circ \delta = \alpha$.
\[
\begin{tikzcd}
A \ar[rd, "\delta"'] \ar[rr, "\alpha"] && B  \\
& C   \ar[ru, "\gamma"'] &
\end{tikzcd}
\]
\end{definition}

Let $A \in \bal$ and $I$ an $\ell$-ideal of $A$. In analogy with the familiar notion of the Jacobson radical, we define the \emph{Jacobson $\ell$-radical} $J_\ell(I)$ of $I$ as the intersection of the maximal
$\ell$-ideals containing $I$; that is,
\[
J_\ell(I) = \bigcap \{ M \in Y_A \mid I \subseteq M\}.
\]
It is easy to see that $J_\ell(I)$ is an $\ell$-ideal, that $I \subseteq J_\ell(I)$, and that $Z_\ell(I) = Z_\ell(J_\ell(I))$.
Consequently, $J_\ell(I)$ is a maximal $\ell$-ideal iff $Z_\ell(I)$ is a singleton. If $A \in \ubal$, then it is known (see, e.g., \cite[Prop.~4.1]{BMO13a}) that $Y_A$ is the set of maximal ideals of $A$, and hence $J_\ell(I)$ is the Jacobson radical of $I$.

\begin{remark}
In general it is not true that $I_\alpha = J_\ell(I_\alpha)$. To see this, consider the compactification $e:X\to Y$ of
Example~\ref{ex: (T) weaker than (N)}. Define $f \in C(Y)$ by $f(n) = 1/n$ and $f(\infty) = 0$. Let $\alpha = e^\flat : C(Y) \to B(X)$.
By Lemma~\ref{lem: remainder}, $Z(I_\alpha) = \{\infty\}$ . Since $\xi_Y : Y \to Y_{C(Y)}$ is a homeomorphism, we have $Z_\ell(I_\alpha) = \{M_\infty\}$. Therefore, $J_\ell(I_\alpha) = M_\infty$, and so $f\in J_\ell(I_\alpha)$.
However, since $\cl_Y(\coz(f)) = Y\not\subseteq e[X]$,
we have that $f \notin I_\alpha$ by Lemma~\ref{lem: characterization of Ialpha}.
\end{remark}

\begin{theorem} \label{thm: 1-point}
Let $X$ be a non-compact completely regular space, $e : X \to Y$ a compactification, and $\alpha = e^\flat : C(Y ) \to B(X)$ the corresponding
basic extension. The following are equivalent.
\begin{enumerate}
\item $\alpha$ is minimal.
\item $X$ is locally compact and $e : X \to Y$ is equivalent to the one-point compactification of $X$.
\item $\alpha$ is locally compact and $J_\ell(I_\alpha)$ is a maximal $\ell$-ideal.
\end{enumerate}
\end{theorem}

\begin{proof}
(1) $\Rightarrow$ (2). We show that $e$ is the least compactification of $X$. Let $e' : X \to Y'$ be an arbitrary compactification
of $X$. Then
$\alpha'=(e')^\flat : C(Y') \to B(X)$ is compatible with $\alpha$.
Since $\alpha$ is minimal, there is $\delta : C(Y) \to C(Y')$ with $\alpha' \circ \delta = \alpha$. By Gelfand-Naimark-Stone duality,
there is a continuous map $\sigma : Y' \to Y$ with $\delta = \sigma^*$, and $\sigma \circ e' = e$ since $\alpha' \circ \delta = \alpha$.
Thus, $e$ is the least compactification of $X$. This, by \cite[Thm.~3.5.12]{Eng89}, yields that $X$ is locally compact and $e$ is equivalent
to the one-point compactification of $X$.

(2) $\Rightarrow$ (1). Suppose that $X$ is locally compact and $e$ is the one-point compactification of $X$. Let $C \in \ubal$ and
$\gamma : C \to B(X)$ be a basic extension compatible with $\alpha$. Then $\gamma_\flat\circ \eta_X : X \to Y_C$ is a compactification.
By \cite[Thm.~3.5.11]{Eng89}, $e$ is the least compactification of $X$, so there is a continuous map $\sigma : Y_C \to Y$ with
$\sigma \circ \gamma_b \circ \eta_X = e$.
\[
\begin{tikzcd}[column sep=5pc]
X_{B(X)} \arrow[r, "\gamma_\flat"] & Y_C \arrow[d, "\sigma"] \\
X \arrow[r, "e"'] \arrow[u, "\eta_X"] & Y
\end{tikzcd}
\]
Therefore, $\sigma^* : C(Y) \to C(Y_C)$ is a morphism in $\bal$. 
\begin{claim} \label{claim}
$\alpha = \widehat{\gamma} \circ \sigma^*$.
\end{claim}

\begin{proofclaim}
By definition, $\widehat{\gamma} = \zeta_{B(X)}^{-1} \circ (\gamma_*)^*$, so
\[
\widehat{\gamma} \circ \sigma^* = \zeta_{B(X)}^{-1} \circ (\gamma_*)^* \circ \sigma^* = \zeta_{B(X)}^{-1} \circ (\sigma \circ \gamma_*)^*.
\]
Thus, it is sufficient to show that $\zeta_{B(X)} \circ e^\flat = (\sigma\circ \gamma_*)^*$. 
\[
\begin{tikzcd}[column sep=5pc]
C(Y) \arrow[r, "e^\flat"] \arrow[d, "\sigma^*"'] \arrow[dr, "(\sigma\circ \gamma_*)^*"] & B(X) \arrow[d, "\zeta_{B(X)}"] \\
C(Y_C) \arrow[r, "(\gamma_*)^*"'] & C(Y_{B(X)})
\end{tikzcd}
\]
Let $f \in C(Y)$. For $x \in X$ we have $\eta_X(x) \in X_{B(X)} \subseteq Y_{B(X)}$. Since $\gamma_\flat = \gamma_*|_{X_{B(X)}}$ and $\sigma \circ \gamma_b \circ \eta_X = e$, we have
\begin{align*}
(\sigma\circ \gamma_*)^*(f)(\eta_X(x)) &= (f\circ\sigma\circ\gamma_*)(\eta_X(x)) =  (f\circ\sigma\circ\gamma_\flat)(\eta_X(x)) \\
&= f\circ (\sigma\circ\gamma_\flat\circ\eta_X)(x) = (f\circ e)(x) \\
&= f(e(x)).
\end{align*}
On the other hand, $(\zeta_{B(X)}\circ e^\flat)(f) = \zeta_{B(X)}(e^\flat(f)) = \zeta_{B(X)}(f \circ e)$. Since $\zeta_{B(X)}(f \circ e)(\eta_X(x))$ is the real number $\lambda$ satisfying $(f\circ e) + \eta_X(x) = \lambda+\eta_X(x)$ and $\eta_X(x) = \{ g \in B(X) \mid g(x)  = 0\}$, we see that $\lambda = f(e(x))$. Consequently, $\zeta_{B(X)}(f \circ e)(\eta_X(x)) = f(e(x))$. This shows that $\zeta_{B(X)}(f\circ e)$ and $(\sigma\circ \gamma_*)^*(f)$ agree on $\eta_X(X)$. Since $X_{B(X)}$ is dense in $Y_{B(X)}$, we conclude that $(\zeta_{B(X)}\circ e^\flat)(f) = \zeta_{B(X)}(f\circ e) = (\sigma\circ \gamma_*)^*(f)$. This yields the claim that $ \alpha = e^\flat = \widehat{\gamma} \circ \sigma^*$. 
\end{proofclaim}

Since $C \in \ubal$,
we have $\zeta_C : C \to C(Y_C)$ is an isomorphism. Set $\delta = \zeta_C^{-1}\circ \sigma^*$. Then $\delta : C(Y) \to C$ is a morphism
in $\bal$. By the definition of $\widehat{\gamma}$ and  Claim~\ref{claim}, $\alpha = \widehat{\gamma} \circ \sigma^* =  \gamma \circ \zeta_C^{-1}  \circ \sigma^* = \gamma \circ \delta$.
Thus, $\alpha$ is minimal.
\[
\begin{tikzcd}[column sep=5pc]
C(Y) \arrow[dd, bend right = 45, "\delta"'] \arrow[rd, "\alpha"] \arrow[d, "\sigma^*"] & \\
C(Y_C) \arrow[d, "\zeta_C^{-1}"] \arrow[r, "\widehat{\gamma}"] & B(X)\\
C \arrow[ur, "\gamma"'] &
\end{tikzcd}
\]

(2) $\Rightarrow$ (3). Suppose that $X$ is locally compact and $e$ is equivalent to the one-point compactification of $X$. Then $\alpha$
is locally compact by Theorem~\ref{thm: local compactness}. By Lemma~\ref{lem: remainder}, $Z(I_\alpha) = \cl_Y(Y \setminus e[X])$. Since
$X$ is locally compact, $e[X]$ is open in $Y$, so $Z(I_\alpha) = Y \setminus e[X]$. Because $e$ is the one-point compactification,
$Y \setminus e[X]$ is a single point, so $J_\ell(I_\alpha)$ is a maximal $\ell$-ideal.

(3) $\Rightarrow$ (2). Suppose $\alpha$ is locally compact and $J_\ell(I_\alpha)$ is a maximal $\ell$-ideal. Since $\alpha$ is locally compact, $X$ is locally compact by
Theorem~\ref{thm: local compactness}.  Because $J_\ell(I_\alpha)$ is a maximal $\ell$-ideal, $Z_\ell(I_\alpha)$ is a singleton. By Remark~\ref{rem: Z versus Zl}, $Z(I_\alpha)$ is a singleton. Therefore, $Y \setminus e[X]$ is a single point by Lemma~\ref{lem: remainder}.
Thus, $e$ is (equivalent to) the one-point compactification of $X$.
\end{proof}

\begin{corollary}
Let $\alpha : A \to B$ be a non-compact basic extension. The following are equivalent.
\begin{enumerate}
\item $\alpha$ is minimal.
\item $X_B$ is locally compact and $\alpha_\flat$ is equivalent to the one-point compactification of $X_B$.
\item $\alpha$ is locally compact and $J_\ell(I_\alpha)$ is a maximal $\ell$-ideal.
\end{enumerate}
\end{corollary}

\begin{proof}
By 
Theorems~\ref{CR duality} and~\ref{thm: 1-point}, it is sufficient to show that $\alpha$ is minimal iff $\widehat{\alpha}$ is minimal.
For this first suppose that $\alpha$ is minimal. Let $\gamma : C \to B$ be a basic extension such that $C \in \ubal$ and $\gamma$ is compatible
with $\widehat{\alpha}$. Since $\alpha$ is compatible with $\widehat{\alpha}$, we have that $\gamma$ is also compatible with $\alpha$. So there
is a morphism $\delta : A \to C$ in $\bal$ with $\gamma \circ \delta = \alpha$. Because $C \in \ubal$ and $\ubal$ is a reflective subcategory of $\bal$ (see, e.g., \cite[p.~447]{BMO13a}), there is a
morphism $\widehat{\delta} : \widehat{A} \to C$ in $\bal$ with $\delta = \widehat{\delta} \circ \zeta_A$. \[
\begin{tikzcd}[column sep=5pc]
A \arrow[dr, "\zeta_A"] \arrow[rr, "\alpha"] \arrow[ddr, bend right = 10, "\delta"'] && B \\
&\widehat{A} \arrow[d, "\widehat{\delta}"'] \arrow[ur, "\widehat{\alpha}"] & \\
& C \arrow[uur, bend right = 10, "\gamma"'] &
\end{tikzcd}
\]
Therefore, $\gamma \circ \widehat{\delta} = \widehat{\alpha}$. Thus, $\widehat{\alpha}$ is minimal.

Conversely, suppose that $\widehat{\alpha}$ is minimal and let $\gamma : C \to B$ be a basic extension such that $C \in \ubal$ and
$\gamma$ is compatible with $\alpha$. Then $\gamma$ is compatible with $\widehat{\alpha}$, so there is a morphism
$\delta : \widehat{A} \to C$ in $\bal$ with $\gamma \circ \delta = \widehat{\alpha}$. Therefore,
$\gamma \circ (\delta \circ \zeta_A) = \widehat{\alpha} \circ \zeta_A = \alpha$. Thus, $\alpha$ is minimal.
\end{proof}

\def\cprime{$'$}
\providecommand{\bysame}{\leavevmode\hbox to3em{\hrulefill}\thinspace}
\providecommand{\MR}{\relax\ifhmode\unskip\space\fi MR }
\providecommand{\MRhref}[2]{%
  \href{http://www.ams.org/mathscinet-getitem?mr=#1}{#2}
}
\providecommand{\href}[2]{#2}

\bigskip

Department of Mathematical Sciences, New Mexico State University, Las Cruces NM 88003,
guram@nmsu.edu, pmorandi@nmsu.edu, bruce@nmsu.edu


\begin{thebibliography}{10}

\bibitem{BMO13a}
G.~Bezhanishvili, P.~J. Morandi, and B.~Olberding, \emph{Bounded {A}rchimedean
  {$\ell$}-algebras and {G}elfand-{N}eumark-{S}tone duality}, Theory Appl.
  Categ. \textbf{28} (2013), Paper No. 16, 435--475.

\bibitem{BMO13b}
\bysame, \emph{Dedekind completions of bounded {A}rchimedean
  {$\ell$}-algebras}, J. Algebra Appl. \textbf{12} (2013), no.~1, 16 pp.

\bibitem{BMO18c}
\bysame, \emph{Canonical extensions of bounded archimedean vector lattices},
  Algebra Universalis \textbf{79} (2018), no.~1, Paper No. 12, 17 pp.

\bibitem{BMO18d}
\bysame, \emph{A generalization of {G}elfand-{N}aimark-{S}tone duality to
  completely regular spaces}, Submitted. Preprint available at
  arXiv:1812.07599, 2018.

\bibitem{BMO19a}
\bysame, \emph{An extension of de {V}ries duality to completely regular spaces
  and compactifications}, Topology Appl. \textbf{257} (2019), 85--105.

\bibitem{Bir79}
G.~Birkhoff, \emph{Lattice theory}, third ed., American Mathematical Society
  Colloquium Publications, vol.~25, American Mathematical Society, Providence,
  R.I., 1979.

\bibitem{BS75}
J.~Blatter and G.~L. Seever, \emph{Interposition of semi-continuous functions
  by continuous functions},  (1975), 27--51. Actualit\'{e}s Sci. Indust., No.
  1367.

\bibitem{Dil50}
R.~P. Dilworth, \emph{The normal completion of the lattice of continuous
  functions}, Trans. Amer. Math. Soc. \textbf{68} (1950), 427--438.

\bibitem{Edw66}
D.~A. Edwards, \emph{Minimum-stable wedges of semicontinuous functions}, Math.
  Scand. \textbf{19} (1966), 15--26.

\bibitem{Eng89}
R.~Engelking, \emph{General topology}, second ed., Sigma Series in Pure
  Mathematics, vol.~6, Heldermann Verlag, Berlin, 1989, Translated from the
  Polish by the author.

\bibitem{GJ60}
L.~Gillman and M.~Jerison, \emph{Rings of continuous functions}, The University
  Series in Higher Mathematics, D. Van Nostrand Co., Inc., Princeton,
  N.J.-Toronto-London-New York, 1960.

\bibitem{Kat51}
M.~Kat{\v{e}}tov, \emph{On real-valued functions in topological spaces}, Fund.
  Math. \textbf{38} (1951), 85--91.

\bibitem{Kat53}
M.~Kat\v{e}tov, \emph{Correction to ``{O}n real-valued functions in topological
  spaces'' ({F}und. {M}ath. 38 (1951), pp. 85--91)}, Fund. Math. \textbf{40}
  (1953), 203--205.

\bibitem{Kub93}
T.~Kubiak, \emph{A stengthening of the {K}at\v{e}tov-{T}ong insertion theorem},
  Comment. Math. Univ. Carolin. \textbf{34} (1993), no.~2, 357--362.

\bibitem{Lan76}
E.P. Lane, \emph{Insertion of a continuous function}, Pacific J. Math.
  \textbf{66} (1976), no.~1, 181--190.

\bibitem{Pic06}
J.~Picado, \emph{A new look at localic interpolation theorems}, Topology Appl.
  \textbf{153} (2006), no.~16, 3203--3218.

\bibitem{Ton52}
H.~Tong, \emph{Some characterizations of normal and perfectly normal spaces},
  Duke Math. J. \textbf{19} (1952), 289--292.

\end{thebibliography}
\end{document}